\documentclass[10pt,A4]{amsart}
\pdfoutput=1
\usepackage{graphicx} 
\usepackage{amsmath, amssymb, stmaryrd, setspace, nicefrac, multicol, amsthm, enumerate}
\usepackage{amsaddr}
\usepackage{xcolor}
\usepackage{tikz}
\usepackage{tcolorbox}
\usepackage{mathrsfs}
\usepackage{subcaption}
\usepackage{comment}
\usepackage{hyperref}
\hypersetup{
	colorlinks,
	linkcolor={blue!50!black},
	citecolor={blue!50!black},
	urlcolor={blue!50!black}
}
\usepackage{cleveref}
\usepackage{todonotes}
\usetikzlibrary{decorations.pathreplacing}
\usetikzlibrary{positioning}
\usetikzlibrary{calc}
\newtheorem{thm}{Theorem}[section]
\newtheorem{lem}[thm]{Lemma}
\newtheorem{cor}[thm]{Corollary}
\newtheorem{prop}[thm]{Proposition}
\theoremstyle{definition}

\newtheorem{defn}[thm]{Definition}
\newtheorem{Q}[thm]{Question}

\newtheorem{rmk}[thm]{Remark}
\newtheorem{ex}[thm]{Example}

\newtheorem{theorem}{Theorem}[section]

\newtheorem{question}[theorem]{Question}

\theoremstyle{definition}
\newtheorem{remark}{Remark}
\newcommand{\eulerone}{\mathscr{E} _{1}}
\newcommand{\eulertwo}{\mathscr{E} _{2}}
\newcommand{\N}{\mathbb{N}}
\newcommand{\Z}{\mathbb{Z}}
\newcommand{\sep}[1]{\mathbf{Sep}({#1})}
\newcommand{\Sep}[1]{\mathbf{Sep}({#1})}
\newcommand{\Ends}[1]{\mathbf{Ends}({#1})}
\newcommand{\Path}[1]{\mathbf{Path}({#1})}
\newcommand{\sepmax}[1]{\mathbf{Sepmax}({#1})}
\newcommand{\HP}{\operatorname{\mathbf{0}} '} 
\renewcommand{\mid}{ : }
\newcommand{\vir}[1]{``#1''}

\newcommand{\comp}{\operatorname{\mathbf{Comp}}}
\newcommand{\G}{G}

\title{Paths, Ends and The Separation Problem for Infinite Graphs}
\author[N. Carrasco-Vargas]{Nicanor Carrasco-Vargas}
\address{Jagiellonian University, Krak\'ow, Poland. \\ {\normalfont \texttt{nicanor.vargas@uj.edu.pl}}}
\author[V. Delle Rose]{Valentino Delle Rose}
\address{Agenzia Nazionale di Valutazione del Sistema Universitario e della Ricerca, Roma, Italy. \\{\normalfont \texttt{v.dellerose@anvur.it}}}
\author[C. Rojas]{Crist\'obal Rojas}
\address{Instituto de Ingenier\'ia Matem\'atica y Computacional, Pontificia Universidad Cat\'olica de Chile, and Centro Nacional de Inteligencia Artificial,  Santiago, Chile.
\\
{\normalfont \texttt{cristobal.rojas@mat.uc.cl}}}
\keywords{Highly computable graphs, number of ends, Eulerian Path Problem for infinite graphs, arithmetical hierarchy, infinite paths}
\begin{document}
\begin{abstract}
We introduce and study the Separation Problem for infinite graphs, which involves determining whether a connected graph splits into at least two infinite connected components after the removal of a given finite set of edges. We prove that this problem is decidable for every highly computable graph with finitely many ends. Using this result, we demonstrate that König’s Infinity Lemma is effective for such graphs. We also apply it to analyze the complexity of the Eulerian Path Problem for infinite graphs, showing that much of its complexity arises from counting ends—indeed, the Eulerian Path Problem becomes strictly easier when restricted to graphs with a fixed number of ends. Under this restriction, we provide a complete characterization of the problem. Finally, we study the Separation Problem in a uniform setting (i.e., where the graph is also part of the input) and offer a nearly complete characterization of its complexity and its relationship to counting the number of ends.
\end{abstract}
\maketitle

\section{Introduction}
The study of infinite graphs from a computability perspective dates back to the seminal work of Manaster and Rosenstein \cite{Manaster-Rosenstein}, as well as Bean \cite{Bean-color, Bean-paths}, who initiated a research program aimed at understanding the algorithmic complexity of deciding classical graph properties for families of graphs with an algorithmic description. The difficulty of deciding a given property depends on the amount of information contained in the graph's description. This research program seeks to characterize this relationship using the hierarchies of algorithmic complexity provided by computability theory.

A central notion in these works is that of \emph{computable graphs}, which are infinite graphs for which there exists an algorithm to determine whether two vertices are connected by an edge. These works also introduced a stronger notion of computability for graphs, known as \emph{highly computable graphs}. Highly computable graphs form a subclass of locally finite computable graphs where, in addition, the degree of each vertex can be effectively computed. Intuitively, a highly computable graph can be thought of as a graph for which there is an algorithm capable of drawing a certified picture of any finite portion of the graph. A third class of graphs, introduced more recently \cite{khoussainov_automatic_1995}, is that of \emph{automatic graphs}. These are graphs whose vertex set and adjacency relation can be described by finite automata. Their algorithmic properties have been systematically compared to those of computable and highly computable graphs \cite{gradel_automatic_2020}.

The main point of considering these stronger notions is that some problems may become easier for graphs with stronger algorithmic descriptions, and an important goal is to understand for which problems this is indeed the case. A particularly interesting example of this phenomenon is the problem of deciding whether a graph has an \emph{Eulerian path}, namely a path which visits every edge exactly once. As it is well-known, the class of graphs admitting an Eulerian path is characterized by the parity of its vertex degrees \cite{Euler,hierholzer1873moglichkeit}. Around two-hundred years later, this characterization was extended to countably infinite graphs by Erd\H{o}s, Gr\"{u}nwald and V\'{a}zsonyi \cite{Erdos-et-al}. Recently, Kuske and Lohrey \cite{Kuske-Lohrey} studied the decidability of the Eulerian Path Problem in this general setting. They proved that the problem is $D_3^0$-complete for computable graphs (where $D_3^0$ denotes the class of sets which are expressible as differences of two $\Sigma_3^0$ sets), but only $\Pi_2^0$-complete for both the class of highly computable graphs and the class of automatic graphs. 

Another problem well known to be undecidable for infinite connected graphs is the problem of finding an infinite simple path. Although K\H{o}nig's lemma guarantees that such paths always exist, they could all be uncomputable even in a highly computable graph \cite{jockusch_pi_1972}.
\subsection{Our contributions} In this paper, we introduce and study what we see as a more fundamental problem underlying many algorithmic questions about connected graphs, including the two previously mentioned. We call it the \emph{Separation Problem}. It consists of determining whether the removal of a finite set of edges $E$ from a connected graph $G$ results in at least two distinct infinite connected components. It is closely related to the problem of counting the number of \emph{ends} of a graph (see \Cref{sec:preliminaries} for a definition). Our main contributions can be summarized as follows. We note that all graphs considered in this paper are assumed to be connected. 

\medskip

\noindent\textbf{On the complexity of the Separation Problem.} We provide a detailed analysis of the Separation Problem and its connection to the number of ends of a graph. In particular: 
        \begin{itemize}
            \item\textbf{For fixed $G$,} we show that the Separation problem is $\Pi_1^0$ in general, and that this bound is tight. However, if the graph is assumed to have finitely many ends, the problem becomes computable;
             \item\textbf{Uniformly in $G$.} To better understand the structure of the uniform version of the Separation Problem when the number of ends is finite, we introduce two specific collections of sets of edges: $\sep{G}$ and $\sepmax{G}$. The first collection contains sets $E$ that separate $G$ into at least two infinite connected components, while the second contains those sets $E$ that separate $G$ in as many distinct infinite connected components as possible. We characterize the complexity of computing each of these collections from a description of $G$, and analyze their relationship to the number of ends. At a high level, our results show that the information contained in $\sepmax{G}$ is uniformly equivalent to the combination of $\sep{G}$ and the number of ends of $G$.
        \end{itemize}
\medskip
\noindent\textbf{Applications.} We apply some of the tools developed in our analysis of the Separation Problem to both the Eulerian path problem and the problem of computing infinite paths, under restrictions on the number of ends of the graph. Our main results are:
        \begin{itemize}
            \item \textbf{On computing infinite paths.} We show that this problem is solvable for graphs with finitely many ends. More precisely, we prove that a connected, highly computable graph with finitely many ends is infinite if and only if it admits a computable infinite simple path. For these graphs we prove the stronger fact that the extension problem for finite paths to infinite simple paths is decidable. 
            \item \textbf{On the Eulerian path Problem.} We consider a version of this problem for graphs with either one or two ends (as graphs with more than two ends cannot have an Eulerian path). We characterise the resulting complexity in each case. In particular, we show that for highly computable graphs with one end, determining whether such a graph has a one-way Eulerian path is d.c.e.-complete (where a set is $d.c.e.$ if it can be written as the difference of two c.e.~set), whereas determining the existence of a two-way Eulerian path is only $\Pi_1^0$-complete. On the other hand, for highly computable graphs with two ends, we show that the problem attains precisely the $m$-degrees of $\Delta_2^0$ sets. Our results on this problem are summarized in \Cref{tab:summary}.           
        \end{itemize}

\begin{table}[h!]
    \centering
    \begin{tabular}{c|c|c}
        {}  & one-way & two-way \\
        \hline
        {} & {} & {} \\
         no bounds on the number of ends &  $\Pi_2^0$-complete &  $\Pi_2^0$-complete  \\
         {} & \cite{Kuske-Lohrey} & \Cref{thm:eulerian-path-general} \\
         
         {} & {} & {} \\
         one end & d.c.e.-complete & $\Pi_1^0$-complete \\
         {} & \Cref{thm:eulerian-path-1end} (1) & \Cref{thm:eulerian-path-1end} (2) \\
         
         {} & {} & {} \\
         two ends & -- & all $m$-degrees of $\Delta_2^0$ sets \\
         {} & {} & \Cref{thm:hardness-e_2-2fin}
    \end{tabular}
    \caption{Hardness of deciding the existence of an Eulerian path in highly computable graphs}
    \label{tab:summary}
\end{table}
We expect that the analysis of the Separation problem that we present here will be of further use in effective graph theory, and in particular in questions related to the computability of paths with different properties. As we shall see in \Cref{sec:infinite-paths}, the Separation problem is, in 
a sense, equivalent to the extension problem for infinite paths if the graphs that we consider are trees. We also mention that a restricted version of the Separation problem was recently used  \cite{nicanor-geometric-subgroup} to prove that infinite Hamiltonian paths are computable on highly computable graphs that are cubes (this is not true in the general case \cite{Bean-paths}), and in \cite{nicanor-on-bean} to show that the extension problem for Eulerian infinite paths is decidable, strengthening a classic result from \cite{Bean-paths}.  

The paper is organized as follows. In \Cref{sec:preliminaries} we briefly review the main notions on computability theory and graph theory that will be needed. In \Cref{section:separation} we introduce the Separation Problem and present a comprehensive analysis of it in terms of the complexity of computing $\sep{G}$, $\sepmax{G}$ and the number of ends of $G$.  In \Cref{sec:infinite-paths} we show that the problem of computing an infinite path can be reduced to membership in $\sep{G}$, and that this reduction becomes an equivalence if $G$ is a tree. We also illustrate some limitations of this result when more specialized paths are considered by giving a counterexample for geodesic infinite paths. Finally, all the results summarized in \Cref{tab:summary} are presented in \Cref{sec:complexity-of-Eulerianity}, along with additional small observations, such as the decidability of the Eulerian Path Problem for one-ended automatic graphs. 

\section{Preliminaries}\label{sec:preliminaries}
\subsection*{Computability theory}
The computability theoretical terminology and notation we use is quite standard and follows mostly the textbook \cite{Soare-book}. For a set $A$ we will write $\overline{A}$ to denote its complement. 

Given a Turing machine $\varphi$ we write $\varphi[s]$ to denote the outcome of running $\varphi$ for $s$ many steps. If $\varphi$ halts within $s$ many steps we write $\varphi[s] \downarrow$, otherwise we write $\varphi[s] \uparrow$. We write $\varphi \downarrow$ if $\varphi[s] \downarrow$ for some $s$, and $\varphi \uparrow$ otherwise. For a c.e.~set $W$, we denote by $W[s]$ the set of its elements enumerated within $s$ steps. We will write $(W_e)_{e\in \N}$ to mean a computable enumeration of all c.e. sets. 

For a given computable enumeration $(\varphi_e)_{e \in \N}$ of Turing machines without input, we denote by $\HP = \{e \colon \, \varphi_e \downarrow\}$ its \emph{halting set}.  Moreover, by means of oracle Turing machines, one can define the halting set \emph{relative to} any set $Y$, i.e.~the set $Y'=\{e: \varphi_e^Y \downarrow\}$. In particular, it is possible to consider $\HP '$, the halting set relative to $\HP$.
In this paper we will make an extensive use of both \emph{Turing reducibility} and $m$-\emph{reducibility} among sets \cite{Soare-book}. The notion of \emph{completeness}, capturing the property of being the hardest problem in a certain class, is defined by means of the latter reducibility: a set $A$ is $\mathcal{C}$-\emph{complete} (for some class of sets $\mathcal{C}$) if $A \in \mathcal{C}$ and $X \le_m A$ for all $X \in \mathcal{C}$. We will mostly set our problems along the levels of the \emph{arithmetical hierarchy}: to this regard, just recall that $\HP$ and $\overline{\HP}$ are, respectively, $\Sigma_1^0$-complete and $\Pi_1^0$-complete, while $\textbf{Inf} = \{e \colon \, |W_e|= \infty \}$ and $\textbf{Fin} = \{e \colon \, |W_e| < \infty \}$ are, respectively, $\Pi_2^0$-complete and $\Sigma_2^0$-complete (see, e.g., \cite[Theorem 4.3.2]{Soare-book}). $\HP'$ is another example of a 
 $\Sigma_2^0$-complete set.

Finally, recall that a set is $\Delta_2^0$ if $A \le_T \HP$.
The following well-known result characterizes the $\Delta_2^0$ sets precisely as those admitting \emph{computable approximations}.
\begin{thm}[Shoenfield's Limit Lemma \cite{limit-lemma}]\label{shoenfield}
    A set $A \subseteq \N$ is $\Delta_2^0$ if and only if $A$ has a \emph{computable approximation}, namely a computable function $f \colon \N^2 \to \{0, 1 \}$ such that, for all $n$:
    \begin{enumerate}
        \item $\lim_{s} f(n,s)$ exists (i.e. $|\{s:f(n,s) \ne f(n,s+1)\}| < \infty$), and
        \item $\lim_{s} f(n,s) = 1 \iff n \in A$.
    \end{enumerate}
\end{thm}
\subsection*{Graph theory}
We will sometimes write \vir{$A \Subset B$} as an abbreviation for \vir{$A$ is finite and $A \subseteq B$}. Unless explicitly mentioned, we consider undirected and simple graphs. The set of vertices of a graph $G$ is denoted by $V(G)$, and its set of edges by $E(G)$. The \emph{degree} $\deg_G(v)$ of a vertex $v$ is the number of edges in $G$ incident to $v$. We say that $G$ is \emph{even} when every vertex has even degree. 
We consider induced subgraphs from sets of vertices and edges. Given $V\subset V(G)$, the \emph{induced} subgraph $G[V]$ has vertex set $V$ and edge set  $\{e\in E(G) : \text{$e$ joins vertices in $V$}\}$. Given $E\subset E(G)$, the \emph{induced} subgraph $G[E]$ has edge set $E$, and vertex set $\{v\in V(G) : \text{$v$ is incident to an edge in $E$}\}$. We denote by $G\smallsetminus E$ the induced subgraph $G[E(G)\smallsetminus E]$. 

A \emph{path} in $G$ is a sequence of vertices $v_0,\dots,v_n$ such that consecutive elements in the sequence are either adjacent or equal, and its \emph{length} is the number of edges that it visits. A path is \textit{simple} when no vertex is repeated. The graph $G$ is \emph{connected} when every pair of vertices can be joined by a path. A connected subgraph of $G$ is a \emph{connected component} when it is maximal for the subgraph relation. Given two vertices $v,w$ in the same connected component, their \emph{distance} $d_G(v,w)$ is the length of the shortest path in $G$ joining them.

The \emph{number of ends} of a connected graph $G$ is the supremum of the number of infinite connected components of $G\smallsetminus E$, where $E$ ranges over all finite subsets of $E(G)$. The number of ends of a graph $G$ may be infinite. Other ways of defining the number of ends of a graph exist, but they all coincide for locally finite graphs \cite{diestel_graphtheoretical_2003}. 

We now recall standard computability notions for countably infinite graphs.
\begin{defn}
    A graph $G$ is \emph{computable} when $V(G)$ is a decidable subset of $\N$, and the adjacency relation is a decidable subset of $\N^2$. A computable graph $G$ is \emph{highly computable} when it is locally finite  and the vertex degree function $\deg_G(\cdot)$ is computable. 
\end{defn}

    We say that a sequence $(G_e)_{e \in \N}$ of highly computable graphs is \emph{uniformly highly computable} if there are computable functions $V$, $E$ and $\deg$ with two inputs, such that $V(e,v)=1$ if and only if $v \in V(G_e)$, $E(e,\langle v,w\rangle)=1$ if and only if $\{v,w\} \in E(G_e)$,  and 
    $\deg(e,v)=n$ if and only if $v$ is a vertex in $G_e$ and has degree $n$ in this graph.  We often identify a highly computable graph $G$ with the index of a Turing machine that computes $V(G)$, $E(G)$, and the function $\deg_G$, and we informally refer to this as a \emph{computer program} for $G$. 
    
    We will also make some observations for the case of automatic graphs, so we end this Section by briefly recalling their definition (see \cite{khoussainov_automatic_1995} for more details). Given an undirected graph $G$, we denote by $R(G)$ the adjacency relation. An undirected graph $G$ is \textbf{automatic} when the relational structure $(V(G),R(G))$ admits an \textbf{automatic presentation}: a tuple $(\Sigma,L,R_e,R_a,h)$ where $\Sigma$ is a finite alphabet, $L\subset \Sigma^\ast$ is a regular language, $h\colon L\to V(G)$ is a surjective function, and $R_e=\{(u,v) : h(u)=h(v)\}, R_a=\{(u,v) : R(G)(h(u),h(v))\}$ are regular relations. 
\section{The Separation problem}\label{section:separation}
In this section we study the Separation problem for infinite, locally finite, and connected graphs. The following definition summarizes the main objects that will be at the center of our attention. 
\begin{defn}\label{def:objects}
Let $G$ be a locally finite and connected graph. 
\begin{enumerate}
    \item $\comp_G(E)$ is the number of infinite connected components of $G\smallsetminus E$, for any $E\Subset E(G)$.  
    \item $\sep{G}$ is the collection of finite sets of edges $E\Subset E(G)$ such that $\comp_G(E)\geq 2$. In this case we say that $E$ \textit{separates} $G$.
    \item $\Ends{G}$ is the number of ends of $G$  (possibly infinite).
    \item $\sepmax{G}$ is the collection of finite sets of edges $E\Subset E(G)$ such that $\comp_G(E)$ attains a maximal value, that is, $\comp_G(E)=\Ends{G}$.
\end{enumerate}
\end{defn}
\begin{rmk} If  $1<\Ends{G}\leq \infty$ then $\sepmax{G}$ is a subset of $\sep{G}$. Also note that if $\Ends{G}=\infty$ then $\sepmax{G}$ is empty. In contrast, if $\Ends{G}=1$, then  $\sep{G}$ is empty while $\sepmax{G}$ contains all finite sets of edges, even the empty one. 
\end{rmk}
Thus the \textit{Separation Problem} corresponds to the membership problem in $\sep{G}$.   
\subsection{The complexity of $\sep{G}$ and $\sepmax{G}$ for a fixed graph $G$}
In this subsection we prove basic results about the computability of $\sep{G}$, $\sepmax{G}$, and the map $E\mapsto\comp_G(E)$, where $G$ is a highly computable connected graph with $1\leq \Ends{G}\leq\infty$.  Some of these results are also uniform on $G$ and will be used later. 

We remark that the strongest computability property that a graph $G$ may posses in this context is having a computable function $E\mapsto\comp_G(E)$. This implies that $\sep{G}$ is a decidable set. If $\Ends{G}<\infty$ then it also follows that $\sepmax{G}$ is decidable (cf. \Cref{def:objects}).

The following result shows that in general  the function $E\mapsto\comp_G(E)$ is \emph{upper-semi-computable}. 
\begin{prop}
\label{prop:comp-is-upper-computable}
There is a partial computable function which on input a (program for a) connected highly computable graph $G$ and a finite set of edges $E\subset E(G)$, computes a  sequence of natural numbers $(\comp_{G,n}(E))_{n\in\N}$ whose minimum equals $\comp_{G}(E)$.
\end{prop}

\begin{proof}
Let $G$ and $E$ be as in the statement. We start by computing the set $V$ of vertices in $G\smallsetminus E$ that are incident to some edge
in $E$. Further, for each $v\in V$, we let $C_{v}$ be the connected
component of $G\smallsetminus E$ that contains $v$. For each $v\in V$ and $n\in\N$, we define $E_{v,n}$
as the set of edges in some path in $G$ with length at most $n$, starting at $v$, and that visits no edge from $E$. The sets $E_{v,n}$ can be uniformly computed (from $G$, $E$,
$v$ and $n$).  It is then clear that for all $u,v\in V$, (1) $C_{v}$ is finite exactly when $E_{v,n}=E_{v,n-1}$ for some $n$, and (2) $C_{u}=C_{v}$ when for some $n$, we have $E_{u,n}\cap E_{v,n}\ne\emptyset$.

Let $\mathcal C _n=\{E_{v,n} \mid v\in V, \ E_{v,n-1}\ne E_{v,n}\}$. We define a relation $\sim$ on $\mathcal C_n$ by  $E\sim E'\iff E\cap E'\ne\emptyset$, and then we extend $\sim$ to its transitive closure. We let $\comp_{G,n}(E)$ be the number of equivalence classes of $\sim$. A direct verification shows that $(\comp_{G,n}(E))_{n\in\N}$ is non-increasing and that $\comp_{G,n}(E)=\comp_G(E)$ for some $n$ large enough.  
\end{proof}
\begin{rmk}\label{comp-function-extra-obs}
    If $\comp_{G,n}(E)=\comp_G(E)$, then $n$ has the property that all connected components of $G\smallsetminus E$ can be determined by ``exploring'' around $E$ at distance at most $n$. 
    
    For a more precise statement, let  $u,v$ be vertices in $G\smallsetminus E$  and incident to $E$. Then they are in the same connected component of $G\smallsetminus E$ if and only if they can be connected with a path in $G\smallsetminus E$ that stays at distance at most $n$ from $E$. 
    Furthermore, $C_u$ is infinite if and only if it contains a vertex $w$ with  $d_{G\smallsetminus E}(w,u)=n$. 
\end{rmk}
We obtain the following general upper bound for the complexity of $\sep{G}$. 
\begin{prop}
\label{prop:decide-only-one-inf-comp} $\sep{G}$ is $\Pi_{1}^{0}$ for every connected and highly computable graph $G$.
\end{prop}
\begin{proof}
    Given $E\Subset E(G)$, we have $E\in\sep{G}$ if and only if for all $n\in\N$ we have $ \comp_{G,n}(E)\geq 2$. This is a $\Pi_1^0$ condition by \Cref{prop:comp-is-upper-computable}.
\end{proof}
This upper bound is tight for graphs with infinitely many ends:
\begin{prop}\label{example-graph-whose-sep-is-halting-complete}
    There is a connected highly computable graph $G$ with infinitely many ends such that $\sep{G}$ is $\Pi_1^0$-complete.   
\end{prop}
\begin{proof}  Let $(\varphi_e)_{e \in \N}$ be a computable enumeration of Turing machines without input, and let $G$ be the graph with vertex set $V(G) = \{ (e,s): \ \varphi_e[s] \uparrow \} \subset \N^2$ and  edge set $E(G) = \{((e,0),(e+1,0): \ e \in \N \} \cup \{((e,s-1),(e,s)): \ \varphi_e[s] \uparrow \}$. 

Intuitively, the graph $G$ looks like a tree in which every branch is associated to a Turing machine, and it is finite or infinite depending on whether the Turing machine halts or not. The following picture illustrates $G$ in the case where $\varphi_0$ and $\varphi_2$ halt but $\varphi_1$ doesn't.
    \begin{tcolorbox}
    \begin{center}
\begin{tikzpicture}
        \node at (0,0) (0) {$\bullet$};
        \node[below = 0 cm of 0] {$(0,0)$};
        \node at (1.5,0) (1) {$\bullet$};
        \node[below = 0 cm of 1] {$(1,0)$};
        \node at (3,0) (2) {$\bullet$};
        \node[below = 0 cm of 2] {$(2,0)$};
        \draw (0) -- (1);
        \draw (1) -- (2);
        \draw (2) -- (4,0);
        \draw[dashed] (4,0) -- (4.75,0);
        \node at (0,1.5) (0') {$\bullet$};
        \node[left = 0 cm of 0'] {$(1,0)$};
        \node at (1.5,1.5) (1') {$\bullet$};
        \node[left = 0 cm of 1'] {$(1,1)$};
        \node at (3,1.5) (2') {$\bullet$};
        \node[left = 0 cm of 2'] {$(2,1)$};
        \node at (0,3) (0'') {$\bullet$};
        \node[left = 0 cm of 0''] {$(2,0)$};
        \node at (1.5,3) (1'') {$\bullet$};
        \node[left = 0 cm of 1''] {$(2,1)$};
        \node at (1.5,4.5) (1''') {$\bullet$};
        \node[left = 0 cm of 1'''] {$(3,1)$};
        \draw (0) -- (0');
        \draw (0') -- (0'');
        \draw (1) -- (1');
        \draw (1') -- (1'');
        \draw (1'') -- (1''');
        \draw (1''') -- (1.5, 5.5);
        \draw[dashed] (1.5,5.5) -- (1.5, 6.25);
        \draw (2) -- (2');
    \end{tikzpicture}
    
    \end{center}
    \end{tcolorbox}
That $\sep{G}$ is $\Pi_1^0$-hard follows from the observation that $\{((e,s),(e,s+1)) : s \in \N\}$ belongs to $\sep{G}$ if and only if $\varphi_e$ does not halt. 
\end{proof}
We now prove that for a graph with finitely many ends, $\comp_G(\cdot)$, $\Ends{G}$ and $\sepmax{G}$ are computable. 
\begin{thm}\label{thm:magic_function}
    Let $G$ be a connected highly computable graph with finitely many ends. Then the map $E\mapsto \comp_G(E)$ is computable. 

    In fact, there is a partial computable function which on input a (program for a) connected highly computable graph $G$ with finitely many ends, a finite set $E \subset V(G)$, the number of ends of $G$ and an arbitrary element $W \in \sepmax{G}$, outputs $\comp_{G}(E)$. 
\end{thm}
\begin{proof}
    We prove the second part of the statement, which implies the first part of the statement. Assume we receive as inputs some description of the highly computable graph $G$, a finite set of edges $W \in \sepmax{G}$, the number $k$ of ends of $G$, and a finite set $E \subset E(G)$. 
    
    In order to compute $\comp_G(E)$ we first compute  a set $U\Subset E(G)$ with the following properties: (1) $U$ contains $E$ and $W$, (2) $G[U]$ is connected, (3) for each pair of vertices  $u,v$ in $G\smallsetminus U$ that are incident to edges from $U$, and such that $u,v$ lie in the same infinite connected component of $G\smallsetminus U$, there is a path from $u$ to $v$ that is completely contained in $G[U\smallsetminus E]$, and (4) $G\smallsetminus U$ has exactly $k$ infinite connected components, and no finite connected component.
    
    Computing a set with these properties requires some care, and we need to introduce some notation. Fix a vertex $v_0$, and for each $r\in\N$ denote by $U_r$ the edge set of the induced graph $G[\{v\in V(G)\mid d_G(v,v_0)\leq r\}]$. We let $L_r$ be the set of vertices in $G\smallsetminus U_r$ that are incident to some edge from $U_r$. Observe that we can compute $U_r$ and $L_r$ from $r$ and the other input parameters, as the graph $G$ is highly computable.    
    
    We  compute $r_0$ as the smallest natural number such that both $E$ and $W$ are contained in $U_{r_0}$.  Thus $U_{r_0}\in\sepmax{G}$ and $\comp_G(E_{r_0})=k$. Since  $k$ is part of the input, we can use \Cref{prop:comp-is-upper-computable} to  compute the smallest natural number $n_0$ such that $\comp_{G,n_0}(U_{r_0})=k$.

    It is clear that $U_{r_0+n_0}$ has properties (1) and (2) above. We prove that it also has property (3). Let $u,v\in L_{r_0+n_0}$ belong to the same infinite connected component of $G\smallsetminus U_{r_0+n_0}$. We remark that by definition of the sets $U_r$ and $L_r$, there is a path contained in $G[U_{r_0+n_0}\smallsetminus U_{r_0}]$ that joins $u$ to some element $u'\in L_{r_0}$. Similarly, there is a path contained in $G[U_{r_0+n_0}\smallsetminus U_{r_0}]$ that joins $v$ to some element $v'\in L_{r_0}$. Since $u$ and $v$ belong to the same infinite connected component of $G\smallsetminus U_{r_0}$, it follows that the same is true for the new elements $u',v'$. Finally, \Cref{comp-function-extra-obs} shows that $u'$ and $v'$ can be joined with a path contained in $G[U_{r_0+n_0}\smallsetminus U_{r_0}]$. Thus there is a path in $G[U_{r_0+n_0}\smallsetminus U_{r_0}]$ that connects $u$ and $v$. 
    
    It only remains to address property (4) above. For this purpose we let $U_f$ be the set of edges in $G$ that lie in some finite connected component of $G\smallsetminus U_{r_0+n_0}$. The set $U_f$ can be easily computed using \Cref{comp-function-extra-obs} (that is, we know that  $U_{r_0+n_0}$ belongs to $\sepmax{G}$ and thus $\comp_G(U_{r_0+n_0})=\Ends{G}$. Using \Cref{prop:comp-is-upper-computable} we can compute some $k_0$ such that $\comp_{G,k_0}(U_{r_0+n_0})=\comp_G(U_{r_0+n_0})$.  \Cref{comp-function-extra-obs} shows that finite  connected components  $G\smallsetminus U_{r_0+n_0}$ are the same as connected components of $G[U_{r_0+n_0+k_0}]$ having no vertex in $L_{r_0+n_0+k_0}$).

      The set $U=U_f\cup U_{r_0+n_0}$ the properties (1), (2), (3) and (4) defined before. We will now use this set to count the number of infinite connected components of $G\smallsetminus E$. Let $V$ be the set of all vertices in $G\smallsetminus U$ that are incident to some edge from $U$, and consider a partition  $V=V_1\sqcup \dots \sqcup V_k$ where $V_i$ is the set of elements in the $i$-th connected component of $G\smallsetminus U$. Our task is to determine for which $i\ne j$, $V_i$ and $V_j$ are in the same connected component of $G\smallsetminus E$.

    We claim that the infinite connected component of $G\smallsetminus E$ that contains $V_i$ equals the one that contains $V_j$ if and only if there is a path within $G[U]$ that visits no edge from $E$, and that joins some vertex in $V_i$ to some vertex in $V_j$. The backward implication is obvious. For the forward implication, suppose that $V_i$ and $V_j$ lie in the same connected component of $G\smallsetminus E$. Then there is a path $p$ in $G$ whose initial vertex is in $V_i$, whose final vertex is in $V_j$, and which visits no edge from $E$. 
    Some segments of $p$ may escape from the graph $G[U\smallsetminus E]$ through a vertex in $V$, and then enter again to $G[U\smallsetminus E]$ through another vertex in $V$. Condition (3) in the definition of $U$ ensures that we can replace these segments by segments completely contained in $G[U\smallsetminus E]$. After replacing all these segments, we end up with a path with the same initial and final vertex as $p$, but contained in $G[U\smallsetminus E]$. This shows the forward implication. 
    
    Define an equivalence relation $\sim$ in $\{1,\dots,k\}$ as follows: $i\sim j$ when there is a path in $G[U]$ that visits no edge from $E$, and that joins some vertex in $V_i$ to some vertex in $V_j$. This relation is computable because $G[U]$ is a finite graph, and we can compute the number of  equivalence classes of $\sim$. The previous paragraph shows that this number equals $\comp_G(E)$.
\end{proof}

\begin{cor}\label{cor:sep-is-decidable-on-graphs-with-finitely-many-ends}
    Let $G$ be a connected highly computable graph with finitely many ends. Then $\sep{G}$ and $\sepmax{G}$ are decidable sets. 
\end{cor}
\begin{proof}
    This follows from \Cref{thm:magic_function} and the facts  $E\in\sep{G}\iff \comp_G(E)\geq 2$,  $E\in\sepmax{G}\iff \comp_G(E)=\Ends{G}$.  
\end{proof}

\subsection{The complexity of $\sep{G}$ and  $\sepmax{G}$ when $G$ is in the input}\label{sec:uniform-sep-ends-and-sepmax}
We have proved that $\sep{G}$ and $\sepmax{G}$ are decidable sets, provided that $G$ is a connected and highly computable graph with finitely many ends. But how difficult is determining membership in $\sep{G}$ and $\sepmax{G}$ when $G$ is also part of the input? The goal of this subsection is showing that these  problems are $\Pi_1^0$-complete and $\Pi_2^0$-complete, respectively. We also prove lower bounds for the problem of finding \textit{any} element in $\sep{G}$ or $\sepmax{G}$, on input $G$. 

The next construction will be used to prove lower bounds associated to $\sep{G}$. 
\begin{ex}\label{ex:grafo-con-colita}
    There is a uniformly highly computable sequence of graphs $(G_e)$, all of them trees with two ends and with vertex set $\Z$, with the following property. For each $e$, let $g_e$ be the edge in $G_e$ that joins $0$ to $1$. Then $\{e : \{g_e\}\in\sep{G_e}\}$ is $\Pi_1^0$-complete.
    
    In order to describe this family of graphs we fix a computable enumeration of Turing machines without input $(\varphi_e)_{e \in \N}$. For each $e$ we define $G_e$ as the graph with vertex set $V(G_e) = \Z$, and with edge relation $E(G_e)$ given by
    \[\{(s,s+1)\mid s\in \N\} \cup 
    \{(-s-1,s)\mid \varphi_{e}[s-1] \uparrow \text{ and } \varphi_{e}[s] \downarrow  \}\cup\]\[
     \{(-s,-s-1)\mid (\varphi_{e}[s-1]\uparrow \text{ and } \varphi_{e}[s] \uparrow) \text{ or } (\varphi_{e}[s-1]\downarrow \text{ and } \varphi_{e}[s] \downarrow)  \}\]
Thus for each $e$ the edges of $G_e$ incident to the vertices $\{-s,\dots,s\}$ can be determined by running $\varphi_e$ for $s$ steps.  For each $s$ such that $\varphi_{e}[s]\uparrow$,  we include edges joining $s$ to $s+1$, and $-s$ to $-s-1$. It follows that if $e$ is such that $\varphi_e\uparrow$, then the graph $G_e$ is simply the graph with vertex set $\Z$, and where consecutive integers are joined. 

The key part of the construction occurs when $\varphi_{e}$ halts in exactly $s$ steps. If $s$ has this property, then we join $s$ to $s+1$, we \textit{do not} join $-s$ to $-s-1$, and we include an edge joining $-s-1$ and $s$. If $s$ is such that $\varphi_{e}$ halts in less than $s$ steps, then we have edges joining $s$ to $s+1$, and $-s$ to $-s-1$ (as in the previous paragraph). Thus  if $\varphi_e$ halts in $s$ steps, then the graph $G_e$ appears to be a bi-infinite line when we look at the segment $\{-s,\dots,s\}$. However, two infinite lines are born from the vertex $s$, while $-s$ is a ``dead end''. This is illustrated in the next picture. 
  \begin{tcolorbox}
        \begin{center}
             \begin{tikzpicture}
                \node at (0,0) (s) {$\bullet$};
                \node[below = 0 cm of s] {$s$};
                \draw (-2,0) -- (s);
                \draw (2,0) -- (s);
                \draw[dashed] (-2.1,0) -- (-2.6,0);
                \draw[dashed] (2.1,0) -- (2.6,0);
                \node at (0,1) (0) {$\bullet$};
                \node[left = 0 cm of 0] {$0$};
                \node at (0,2) (-s) {$\bullet$};
                \node[left = 0 cm of -s] {$-s$};
                \draw (s) -- (0) -- (-s);
            \end{tikzpicture}
        \end{center}
    \end{tcolorbox}
It follows from the construction that  the edge  $g_e$ in $G_e$ joining $0$ to $1$ satisfies $\{g_e\}\in\sep{G_e}$ if and only if $\varphi_e\uparrow$.
\end{ex} 
\begin{prop}\label{complexity-of-sep-uniform-version}
The problem of given a highly computable and connected graph $G$ and $E\Subset E(G)$, to determine whether $E\in \Sep{G}$, is $\Pi_1^0$-complete.
\end{prop}
\begin{proof}
    The proof of \Cref{prop:decide-only-one-inf-comp} is uniform and this provides the upper bound. The lower bound follows from \Cref{ex:grafo-con-colita}. 
\end{proof}
The same construction shows that any function which on input (a computer program for) $G$ finds a \textit{witness of separation} must compute the halting problem. 
\begin{prop}\label{computing-one-element-in-sep-is-Turing-complete}
Let $f$ be a function such that on input (a computer program for) the connected and highly computable graph $G$, $f$ outputs an element in $\sep{G}$. Then $f\geq_T\HP$.
\end{prop}
\begin{proof}
    Let $(G_e)_{e\in\N}$ and  $(\varphi_e)_{e\in\N}$ be as in \Cref{ex:grafo-con-colita}. For each $e$, we use $f$ to compute an element $E_e$ in $\sep{G_e}$ and we let $s_e\in\N$ be the maximum value of $|v|$, where $v\in \Z$ ranges over vertices incident to edges in $E_e$. Observe that $f\geq_T (s_e)_{e\in\N}$. Furthermore the following holds for each $e$: if $\varphi_e$ halts, then this happens in at most $s_e$ steps.  Thus $(s_e)_{e\in\N}\geq_T\{e\in\N : \varphi_e\downarrow\}$ and therefore $f\geq_T\HP$.
\end{proof}
In the case of $\sepmax{G}$ our lower bounds will be obtained from the next example. 
\begin{ex}\label{grafo-con-ciclos}
There is a uniformly highly computable family of graphs $(G_e)_{e\in\N}$, all of them with either $1$ or $2$ ends, such that $\{e\in\N : \Ends{G_e}=1\}$ is $\Pi_2^0$-complete. 

In order to describe this family of graphs we fix a computable enumeration of all c.e. sets $(W_e)_{e \in \N}$. We recall that  $\textbf{Inf}=\{e\in\N:|W_e|=\infty\}$ is $\Pi_2^0$-complete. For every $e$ the graph $G_e$ has vertex set $\Z$, and its edge set is constructed according to the following effective procedure. We start enumerating the set $W_e$, and as long as $W_{e}[s] = W_{e}[s-1]$, we put edges between consecutive integers in $\{-s-1,\dots,s+1\}$. Thus the graph looks like a line: \begin{tcolorbox}
        \begin{center}
            \begin{tikzpicture}
                \node at (-1,0) (-1) {$\bullet$};
                \node at (0,0) (0) {$\bullet$};
                \node at (1,0) (1) {$\bullet$};
                \node at (-4,0) (-s) {$\bullet$};
                \node at (4,0) (s) {$\bullet$};
                \node at (-5,0) (-s-1) {$\bullet$};
                \node at (5,0) (s+1) {$\bullet$};
                \node[above =0 cm of -1] {$-1$};
                \node[above =0 cm of 0] {$0$};
                \node[above =0 cm of 1] {$1$};
                \node[above =0 cm of -s] {$-s$};
                \node[above =0 cm of s]{$s$};
                \node[above =0 cm of -s-1] {$-s-1$};
                \node[above =0 cm of s+1]  {$s+1$};
                \draw (-1) -- (0);
                \draw (0) -- (1);
                \draw (-1) -- (-2,0);
                \draw[dashed] (-2.1,0) -- (-2.9,0);
                \draw (-3,0) -- (-s);
                \draw (1) -- (2,0);
                \draw[dashed] (2.1,0) -- (2.9,0);
                \draw (3,0) -- (s);
                \draw[blue, thick] (-s) -- (-s-1);
                \draw[blue, thick] (s) -- (s+1);
                \draw[dashed] (-s-1) -- (-6,0);
                \draw[dashed] (s+1) -- (6,0);
            \end{tikzpicture}
        \end{center}
    \end{tcolorbox}
    As soon as we see that $W_{e}[s] \ne W_{e}[s-1]$ (that is, a new element enters $W_e$ at stage $s$), we omit the edge between $-s$ and $-s-1$ and, instead, connect $s$ to $-s$ and also to both $s+1$ and $-s-1$. In this manner we create a cycle with vertex set $\{-s,\dots,s\}$.

    \begin{tcolorbox}
        \begin{center}
            \begin{tikzpicture}
                \node at (-1,0) (-s-1) {$\bullet$};
                \node at (0,0) (s) {$\bullet$};
                \node at (1,0) (s+1) {$\bullet$};
                \node[below=0 cm of -s-1] {$-s-1$};
                \node[below =0 cm of s] {$s$};
                \node[below=0 cm of s+1] {$s+1$};
                \node at (.75,0.5) (s-1) {$\bullet$};
                \node at (-.75,0.5) (-s) {$\bullet$};
                \node[right =0 cm of s-1] {$s-1$};
                \node[left =0 cm of -s] {$-s$};
                \draw[blue, thick] (s) -- (s+1);
                \draw[blue, thick] (s) -- (-s-1);
                \draw[blue, thick] (s) -- (-s);
                \draw (s) -- (s-1);
                \node at (.75, 1.25) (1) {$\bullet$};
                \node at (0, 1.75) (0) {$\bullet$};
                \node at (-.75, 1.25) (-1) {$\bullet$};
                \node[left = 0 cm of -1] {$-1$};
                \node[above = 0cm of 0] {$0$};
                \node[right = 0cm of 1] {$1$};
                \draw (0) -- (-1);
                \draw (0) -- (1);
                \draw[thick, dashed] (-s) -- (-1);
                \draw[thick, dashed] (s-1) -- (1);
                \draw[dashed] (-s-1) -- (-2,0);
                \draw[dashed] (s+1) -- (2,0);
            \end{tikzpicture}
        \end{center}
    \end{tcolorbox}
    
This process is iterated, and more cycles will be created whenever new elements enter $W_e$. According to this procedure, exactly one cycle is created for each element $W_e$. 
\begin{tcolorbox}            
        \begin{center}               
            \begin{tikzpicture}
                \node at (0,0) (0) {$\bullet$};
                \node at (1,0) (s0) {$\bullet$};
                \node at (2,0) (sk) {$\bullet$};
                \node at (3,0) (s*) {$\bullet$};
                \draw (0) .. controls (0.25, 0.5) and (0.75, 0.5) .. (s0);
                \draw (0) .. controls (0.25, -0.5) and (0.75, -0.5) .. (s0);
                \draw[dashed] (s0) -- (sk);
                \draw (sk) .. controls (2.25, 0.5) and (2.75, 0.5) .. (s*);
                \draw (sk) .. controls (2.25, -0.5) and (2.75, -0.5) .. (s*);
                \draw (s*) -- ++(20:0.75cm);
                \draw[dashed] (s*)++(20:0.75cm) -- ++(20:0.75cm);
                \draw (s*) -- ++(-20:0.75cm);
                \draw[dashed] (s*)++(-20:0.75cm) -- ++(-20:0.75cm);
                \node[below = 0 cm of 0] {$0$};
                \node[below = 0 cm of s0] {$s_0$};
                \node[below = 0 cm of sk] {$s_k$};
                \node[below = 0 cm of s*] {$s^*$};
                \coordinate (U) at ($(s*)+(20:1.5cm)$);
                \coordinate (B) at ($(s*)+(-20:1.5cm)$);
                \node (F) at ($0.5*(U)+0.5*(B)$) {};
                \node (L) at ($0.5*(0)+0.5*(F)$) {};
                \node[below = 0.75cm of L] {$G_e$ for $e \in \textbf{Fin}$};
            \end{tikzpicture}
            \qquad \qquad
            \begin{tikzpicture}
                \node at (0,0) (0) {$\bullet$};
                \node at (1,0) (s0) {$\bullet$};
                \node at (2,0) (s1) {$\bullet$};
                \node at (3,0) (s2) {$\bullet$};
                \draw (0) .. controls (0.25, 0.5) and (0.75, 0.5) .. (s0);
                \draw (0) .. controls (0.25, -0.5) and (0.75, -0.5) .. (s0);
                \draw (s0) .. controls (1.25, 0.5) and (1.75, 0.5) .. (s1);
                \draw (s0) .. controls (1.25, -0.5) and (1.75, -0.5) .. (s1);
                \draw (s1) .. controls (2.25, 0.5) and (2.75, 0.5) .. (s2);
                \draw (s1) .. controls (2.25, -0.5) and (2.75, -0.5) .. (s2);
                \draw[dashed] (s2) -- (4,0);
                \node[below = 0 cm of 0] {$0$};
                \node[below = 0 cm of s0] {$s_0$};
                \node[below = 0 cm of s1] {$s_1$};
                \node[below = 0 cm of s2] {$s_2$};
                \node[below = 0.65 cm of s1] {$G_e$ for $e \in \textbf{Inf}$}; 
            \end{tikzpicture}
        \end{center}
    \end{tcolorbox}
It is clear that $G_e$ has one end if and only if $W_e$ is infinite.
\end{ex}

\begin{prop}\label{Determining-membership-in-sepmax-uniformly-is-Pi2}The problem of given a highly computable and connected graph $G$ and $E\Subset E(G)$, to determine whether $E\in \sepmax{G}$, is $\Pi_2^0$-complete. 
\end{prop}
\begin{proof}
We have that $E\Subset E(G)$ lies in $\sepmax{G}$ if and only if $(\forall F\Subset E(G),\forall n\in\N)(\exists m\geq n)(\comp_{G,m}(F)\leq \comp_{G,m}(E))$, which is a $\Pi_2^0$ condition by \Cref{prop:comp-is-upper-computable}.  For the lower bound we consider a minor modification of the sequence $G_e$ defined in  \Cref{grafo-con-ciclos}. For each $e$, let $N_e$ be the graph whose vertex set is a copy of the naturals $\{n' : n\in\N\}$, and where $n'$ is joined to $m'$ if and only if $n,m$ are consecutive integers. For each $e$ we let $H_e$ be the graph obtained by taking the union of $N_e$ and $G_e$, and including one edge $g_e$ joining the vertex $0$ of $G_e$ to the vertex $0'$ of $N_e$. Then $\{g_e\}\in\sepmax{H_e}$ if and only $G_e$ has one end, and thus $\{e : \{g_e\}\in\sepmax{H_e}\}$ is a $\Pi_2^0$-complete set. 
\end{proof}
\begin{prop}\label{prop:lower-bound-sepmax}
Let $f$ be any function such that on input (a computer program for) the connected and highly computable graph $G$, $f$ outputs an element in $\sepmax{G}$. Then $f\geq_T\HP'$.
\end{prop}
\begin{proof}    
    We consider the family of graphs $(G_e)_{e\in\N}$  from \Cref{grafo-con-ciclos}. Since the set $\{e : \Ends{G_e}=1\}$ is $\Pi_2^0$-complete (see \Cref{grafo-con-ciclos}), we also have $\{e : \Ends{G_e}=1\}\geq_T\HP'$, so it suffices to describe a process which having oracle access to $f$ determines the number of ends of $G_e$. 
    
    On input $e$ we use $f$ to compute an element $E_e$ in $\sepmax{G_e}$. Note that if $\Ends{G_e}=2$ then we have $E_e\in\sep{G_e}$ (as in this case $\sep{G}=\sepmax{G})$, while if  $\Ends{G_e}=1$ then  $E_e\not\in\sep{G_e}$ (being the set $\sep{G}$ empty if $G$ is a graph with one end). It follows that $\Ends{G_e}=1$ if and only if $E_e\not\in\sep{G_e}$, which can be determined with oracle $\HP$  (\Cref{complexity-of-sep-uniform-version}). Therefore if we prove that $f\geq_T\HP$, then we will also have  $f\geq_T\HP'$.

    It suffices to observe that the lower bound in \Cref{computing-one-element-in-sep-is-Turing-complete}  was proved with a family of graphs with two ends. Since $\sep{\cdot}$ and $\sepmax{\cdot}$ coincide in this case, it follows that $f\geq_T\HP$ as desired.
\end{proof}
\Cref{grafo-con-ciclos} provides an alternative proof of the following known result (see \cite{Kuske-Lohrey}).
\begin{prop}
    Let $k\geq 1$. Then the problem of determining whether a highly computable and connected graph has at most $k$ ends is $\Pi^0_2$-complete.
\end{prop}
\begin{proof}
    Let $G$ be a graph as in the statement. Then $G$ has at most $k$ ends if and only if $(\forall E\Subset E(G))(\exists n\in\N)(\comp_{G,n}(E)\leq k)$, which is a $\Pi_2^0$ condition by \Cref{prop:comp-is-upper-computable}. The lower bound in the case $k=1$ is proved in \Cref{grafo-con-ciclos}. The case $k> 1$ is proved by a trivial modification of the same family. That is, it suffices to attach to each $G_e$ $k-1$ extra infinite lines. 
\end{proof}

\subsection{The relationship between $\sep{G}$, $\sepmax{G}$ and  $\Ends{G}$}
We have shown that the problem of determining membership in $\sepmax{G}$, with an algorithm uniform on $G$, is a $\Pi_2^0$-complete problem. Since this coincides with the complexity of bounding the number of ends of $G$ from above, the following questions arise naturally:
\begin{quote}
    Suppose we are given a highly computable and connected graph $G$ with finitely many ends, and we are given oracle access to $\sepmax{G}$. Is this information sufficient to determine  $\Ends{G}$? What about the converse?
\end{quote}
In this section we study these and related questions. Our main result states that the knowledge of $\sepmax{G}$ is in a sense equivalent to knowledge of both $\sep{G}$ and $\Ends{G}$:
\begin{thm}\label{ends+sep=sepmax}
    Let $G$ be a highly computable graph with finitely many ends. Then:
    \begin{enumerate}
        \item There is an algorithm, uniform in $G$, which with oracle access to $\sepmax{G}$ computes $\Ends{G}$ and determines membership in $\sep{G}$.
        \item There is an algorithm, uniform in $G$, which given $\Ends{G}$ and with oracle access to $\sep{G}$, decides membership in $\sepmax{G}$. 
    \end{enumerate}
\end{thm}

Our main tool will be the following elementary result.
\begin{lem}\label{minimal-for-inclusion-sets}
    Let $G$ be an infinite, connected, and locally finite graph with finitely many ends, and let $w$ be a vertex in $G$. For each $r\in\N$ let $V_r$ be the set of vertices in $G$ at distance $r$ to $w_0$. Furthermore, we let $E_r$ be the set of edges joining some vertex in $V_r$ to some vertex in $V_{r-1}$.
    \begin{enumerate}
    \item Suppose that $E_r\in\sepmax{G}$. Then $E\subset E_r$ is minimal for inclusion within $\sepmax{G}$ if and only if either $\Ends{G}=1$ and $E=\emptyset$, or $\Ends{G}\geq2$ and there is an  infinite connected component $C$  of $G\smallsetminus E$ such that $E$ is equal to all edges in $E_r$ incident to a vertex in some infinite connected component of $G\smallsetminus E$ different from $C$.   
    
    \item Suppose that  $E_r\in\sep{G}$. Then $E\subset E_r$ is  minimal for inclusion within $\sep{G}$ if and only if there is an infinite connected component $C$ of $G\smallsetminus E_r$, such that $E$ equals the set of all edges in $E_r$ incident to some vertex from $C$. 
    \item $E_r$ belongs to $\sepmax{G}$ if and only if there are exactly $k$ disjoint subsets of $E_r$ that are minimal for inclusion within $\sep{G}$.
    \end{enumerate}
\end{lem}
\begin{proof}
Observe that the three claims hold for trivial reasons if $\Ends{G}=1$. Suppose in what follows that $\Ends{G}>1$. Let $r\geq 1$, and let us introduce some useful notation regarding the connected components of $G\smallsetminus E_r$. Let  $C_1,\dots,C_k$ be the infinite connected components of $G\smallsetminus E_r$, $k\geq 1$. For each $i=1,\dots,k$ we let $E_r^i$ be the set of edges in $E_r$ that are incident to some vertex from $C_i$.  

Let us suppose $E_r\in\sepmax{G}$ and prove that (1) holds. Let $E\subset E_r$ be minimal for inclusion within $\sepmax{G}$. It is clear that then $E$ contains no edge incident to a finite connected component of $G\smallsetminus E_r$. Furthermore,  there can not be two different components $C_i$ and $C_j$, $i,j\in\{1,\dots,k\}$, such that $E$ fails to contain both $E_r^i$ and $E_r^j$. Indeed, in this situation the graph $G\smallsetminus E$ would have a single component containing both $C_i$ and $C_j$, therefore a component with two ends, contradicting the assumption that $E\in\sepmax{G}$. It follows that $E$ must contain at least $k-1$ of the sets $E_r^i$, $i=1,\dots,k$. Let $i_0\in\{1,\dots,k\}$ be such that $E$ contains $E'=\bigcup_{1\leq i\leq k,\ i\ne i_0}E_r^i$. It is straightforward that $E'\in\sepmax{G}$. Since $E$ was assumed to be minimal for inclusion within $\sepmax{G}$, it follows that $E=E'$ and the claim is verified. The backward implication in (1) follows the same arguments.  

Let us suppose $E_r\in\sep{G}$ and prove that (2) holds. Let $E\subset E_r$ be minimal for inclusion within $\sep{G}$. Then there must be some $i\in\{1,\dots,k\}$ such that $E$ contains $E_r^i$, for otherwise $G\smallsetminus E$ would have exactly one infinite connected component and this is not possible for $E_r\in\sep{G}$. As in the previous paragraph, minimality implies that in fact $E=E_r^i$, and the claim is verified. The backward implication for (2) follows similar arguments. Finally, observe that (3) follows directly from (2).\end{proof}

\begin{proof}[Proof of \Cref{ends+sep=sepmax}]
We start with the first claim in the statement. Suppose we are given as input a (computer program for a) highly computable and connected graph $G$ with finitely many ends, and oracle access to $\sepmax{G}$. We start showing that this is sufficient to compute $\Ends{G}$. For each $r\geq 1$ we let $E_r$ be as in \Cref{minimal-for-inclusion-sets}. Having oracle access to $\sepmax{G}$ we can compute the smallest natural number $r_0$ such that $E_{r_0}\in\sepmax{G}$. Next we compute the number of finite subsets $E\subset E_{r_0}$ that also belong to $\sepmax{G}$ and that are minimal for inclusion in $\sepmax{G}$. This is equal to $\Ends{G}$ by \Cref{minimal-for-inclusion-sets} (this is also true if $\Ends{G}=1$, in which case $r_0=1$ and $\emptyset$ is the only minimal for inclusion subset of $E_r$ in $\sepmax{G}$). Thus we have computed $\Ends{G}$. Having the information of $E_{r_0}$ and $\Ends{G}$, we can decide membership in $\sep{G}$ with the function from \Cref{thm:magic_function}.

We now prove the second claim in the statement. Suppose we are given as input a (computer program for a) highly computable and connected graph $G$ with finitely many ends, the value $k=\Ends{G}$, and oracle access to $\sep{G}$.  For each $r\geq 1$ we let $E_r$ be as in \Cref{minimal-for-inclusion-sets}. We let $r=r_0$ be the smallest natural number such that $E_r$ has exactly $k$ disjoint subsets that are minimal for inclusion within $\sep{G}$. The value $r_0$ must exist by the third item in \Cref{minimal-for-inclusion-sets} and because $E_r\in\sepmax{G}$ for all $r$ large enough. Furthermore, we can compute $r_0$ having oracle acccess to $\sep{G}$. Having the information of $E_{r_0}\in\sepmax{G}$ and $k=\Ends{G}$, we can decide membership in $\sepmax{G}$ with the algorithm in \Cref{thm:magic_function}.  
\end{proof}

We now provide some examples showing that  \Cref{ends+sep=sepmax} cannot be improved significantly. The following example shows that  $\Ends{G}$ is not enough to compute $\sep{G}$, even if the graphs are assumed to be trees. 
\begin{prop}
    There is a uniformly highly computable sequence of trees $(G_e)_{e\in\N}$ such that $\Ends{G}$ is uniformly computable, but $\sep{G}$ is not. 
 \end{prop}
\begin{proof}
    This is proved by \Cref{ex:grafo-con-colita}.
\end{proof}
Similarly, we observe that having access to $\sep{G}$ is not enough to compute $\Ends{G}$, even if the graphs are assumed to be trees. 
\begin{prop}
    There is a uniformly highly computable sequence of trees $(G_e)_{e\in\N}$ such that $\sep{G}$ is uniformly computable, but $\Ends{G}$ is not.
\end{prop}

\begin{proof} Fix a computable enumeration $(\varphi_e)_{e \in \N}$ of Turing machines without input. Let $V(G_e) = \Z \cup \{(s,1) \mid \varphi_{e}[s-1] \downarrow \}$. Two vertices in $V(G_e)$ are joined by an edge if they are horizontally adjacent (according to $\Z^2$), and we also include an edge joining $(s,0)$ to $(s,1)$ if $s$ is such that $\varphi_{e}[s]\downarrow$ and  $\varphi_{e}[s-1]\uparrow$. It is clear that $\sep{G_e}$ is uniformly decidable since  for each $e$, $G_e$ is a tree without finite branches and therefore $\sep{G_e}$ contains all finite sets of edges. On the other hand, an ``additional end'' will appear exactly when $\varphi_e$ halts, and therefore $\Ends{G_e}$ is not computable on input $e$. 
\end{proof}
\section{Computing infinite paths by separation}\label{sec:infinite-paths}

Here we establish a close connection between the separation problem, and the problem of determining whether a finite simple path in a graph can be extended to an infinite one. We  denote by $\Path{G}$ the set of finite sequences  $(v_i)_{i=0}^{n}$ of vertices in $V(G)$ with the property that $(v_i)_{i=0}^{n}$ can be extended to an infinite simple path $(v_n)_{n\in\N}$.


\begin{thm}
\label{thm:sep-vs-path}Let $G$ be a highly computable and connected infinite 
graph. Then $\sep{G}\geq_{T}\Path{G}$. If in addition $G$ is a tree, then
we have an equivalence $\sep{G}\equiv_{T}\Path{G}$. The reductions
are uniform on (a computer program for) $G$. 
\end{thm}
Putting this result together with \Cref{cor:sep-is-decidable-on-graphs-with-finitely-many-ends} we obtain the following amusing consequence. 

\begin{thm}\label{paths-is-decidable-for-finitely-many-ends}
     If $G$ is a highly computable infinite graph with finitely many ends, then $\Path{G}$ is a decidable set. In particular,  $G$ admits computable infinite simple paths. 
\end{thm}
Thus for a highly computable and connected graph $G$, the property ``having infinite simple paths but no computable one'' implies that $G$ has infinitely many ends, and that $\sep{G}$ is undecidable. The proof of \Cref{thm:sep-vs-path} will be based on the following result. 

\begin{lem}\label{sep-y-vecinos}
    Let $G$ be a highly computable connected graph. Let $P$ be the set of pairs $(E,v)$ such that $E\subset E(G)$ is finite, $v$ is a vertex in $G\smallsetminus E$ incident to some edge from $E$, and the connected component of $G\smallsetminus E$ containing $v$ is infinite. Then $P\leq_T\sep{G}$. This Turing reduction is uniform on (a computer program for) $G$.
\end{lem}
\begin{proof}
Let $E\subset E(G)$ be finite, let $v$ be a vertex in $G\smallsetminus E$ incident to some edge from $E$, and let $C_v$ be the connected component of $G\smallsetminus E$ containing $v$. It suffices to describe an algorithm which, having oracle access to $\sep{G}$, determines whether $C_v$ is finite or infinite. 

We start by checking whether $E\in \sep{G}$. If $E\not\in\sep{G}$ then  $\comp_{G}(E)=1$. We can use \Cref{prop:comp-is-upper-computable} to compute $n\in\N$ such that $\comp_{G,n}(E)=\comp_{G}(E)$. Then \Cref{comp-function-extra-obs} provides a criterion to determine whether $C_v$ is infinite.

Suppose now that $E\in\sep{G}$. For each $r\in\N$ we define $E_{v,r}$ as the set of edges in $E(G)$ that belong to a path in $G\smallsetminus E$ which visits $v$ and has length at most $r$. An elementary argument shows that $C_v$ is infinite if and only if for some $r\in \N$ we have $E_{v,r}\in\sep{G}$. Thus ``$C_v$ is infinite'' is a $\Sigma_1^0$ property with oracle $\sep{G}$. On the other hand, it is clear that ``$C_v$ is finite'' is a $\Sigma_1^0$ condition. This provides an algorithm which, having $\sep{G}$ as oracle, determines whether $C_v$ is infinite.
\end{proof}

\begin{proof}[Proof of \Cref{thm:sep-vs-path}]
Let $(v_i)_{i=0}^{n}$ be a finite sequence of vertices in $G$. Since $G$ is highly computable we can check whether it corresponds to a finite simple path, that is, whether each vertex is adjacent to the next one and whether no vertices are repeated. Let $E$ be the set of all edges in $G$ that are incident to $v_i$ for some $i=1,\dots,n$, and let $U$ be the set of vertices in $G\smallsetminus E$ that are adjacent to $v_n$. By König's Infinity Lemma, $(v_i)_{i=0}^{i=n}$ can be extended to an infinite simple path in $G$ if and only if there is $u\in U$ such that $C_u$ is infinite. This is decidable with oracle $\sep{G}$ by \Cref{sep-y-vecinos}. We have proved  the reduction $\sep{G}\geq_{T}\Path{G}$. 

We now assume that $\G$ is a tree, and prove the reduction $\Path{G}\geq_{T}\sep{G}$.
We exhibit a procedure which, given a finite set $E\subset E(\G)$, a
vertex $u$ in $\G\smallsetminus E$ that is incident to some edge
in $E$, and using $\Path{G}$ as oracle, determines if the connected
component of $\G\smallsetminus E$ containing $u$ is infinite. It
is clear how to use this information to check whether $E\in\sep{G}$. The procedure is as follows. Let $C_u$ be the connected component of $\G\smallsetminus E$ containing $u$. We compute the
set $V$ of vertices in  $\G\smallsetminus E$ adjacent to $u$. 
If $V$ is
empty then we conclude that $C_u$ is  finite. If $V$ is
nonempty, then we observe that, by König's infinity Lemma, $C_u$
is infinite if and only if it has some infinite path starting at $u$. These paths are in correspondence
with infinite paths in $\G$ that start at $u$ and then visit
some vertex $v\in V$. This is true because both $C$ and $\G$ are trees. Thus in order to decide whether
$C$ is infinite, it suffices to check whether some of the finitely
many elements $\{uv\mid v\in V\}$ lies in $\Path{G}$. 
\end{proof}

We now present an example illustrating the limitations of \Cref{thm:sep-vs-path}. A finite path $(v_i)_{i=n}^{i=m}$ in a graph is called \textit{geodesic} when its length is minimal among all finite paths with the same initial and final vertex, and an infinite path is geodesic when its restriction to $\{n,\dots,m\}$ is geodesic for all $n,m\in\N$, $n<m$. This concept is fundamental in geometric group theory, see for instance \cite[Chapter 11]{drutu_geometric_2018}.

\begin{prop}\label{graph-with-no-computable-infinite-geodesic-paths}
    There is a highly computable graph $\Lambda$ with one end and with no computable infinite geodesic path. 
\end{prop}
\begin{proof}
    Let $T$ be a subtree of the rooted binary tree that is connected, highly computable, infinite, and with no computable infinite simple path. The existence of $T$ is proved in \cite{jockusch_pi_1972}. Let $\Lambda$ be the simple graph whose vertex set is the cartesian product $V(T)\times V(T)$, and where two vertices $(u_1,v_1)$ and $(u_2,v_2)$ are neighbors when either (1) $u_1=u_2$ and $v_1$ is adjacent to $v_2$ in $T$, or (2) $v_1=v_2$ and $u_1$ is adjacent to $u_2$ in $T$. If $\Lambda$ had a computable infinite geodesic path, then one could take the projection to the first or second coordinate and obtain a computable infinite simple path on $T$, which is a contradiction.  The elementary verification is left to the reader.  
\end{proof}
\begin{rmk}
Notice that even though the graph $\Lambda$ above does not have any computable geodesic infinite path, it must have \emph{some} computable infinite simple path, by \Cref{paths-is-decidable-for-finitely-many-ends}.  What do such paths look like then? The key observation is that, if we think of a path in $\Lambda$ as a combination of two paths $p_1$ and $p_2$ in the tree $T$, then one can search for an infinite path in $T$ using $p_1$ while keeping $p_2$ at a fixed node, and in case we chose the wrong branch and end up in a leaf of $T$, we can change the node from $p_2$ and then “go back'' with $p_1$ to try a different branch, without formally repeating nodes in $\Lambda$.  
\end{rmk}
We finish this section with some  questions that we were unable to solve. The next question raises naturally from \Cref{thm:sep-vs-path}.
\begin{Q}
    Is it true that $\Path{G}\equiv_T\sep{G}$ for every highly computable and connected graph $G$? 
\end{Q}
The next questions  
were communicated to us by N. Bitar and V. Salo during the 2024 conference ``Complexity of Simple Dynamical Systems" at CIRM, and arise naturally in the study of symbolic dynamics on groups. A positive answer would allow to strengthen \cite[Proposition 1.4]{aubrun_self-avoiding_2024}.
\begin{question}
    Let $G$ be a highly computable graph with a decidable separation problem. Is it decidable whether a finite path in $G$ can be extended to a two-way infinite path on $G$? 
\end{question}
\begin{question}
    Let $G$ be a highly computable graph, and suppose that $G$ is a Cayley graph of a group with decidable word problem. Is the separation problem for $G$ decidable?  
\end{question}

\section{On the complexity of the Eulerian path problem}\label{sec:complexity-of-Eulerianity}

Erd\H{o}s, Gr\"{u}nwald and V\'{a}zsonyi established the following characterization of infinite \emph{Eulerian} graphs, namely those graphs admitting an Eulerian path \cite{Erdos-et-al}.

\begin{thm}[\cite{Erdos-et-al}] \label{egw}
Let $G$ be a connected graph with countably many edges. 
\begin{enumerate}[(i)]
    \item $G$ admits a one-way infinite Eulerian path if and only if it satisfies the following set of conditions, called $\eulerone$:
    \begin{itemize}
        \item Either $G$ has exactly one vertex with odd degree, or $G$ has at least one vertex with infinite degree and no vertex with odd degree.   
        \item $G$ has one end.
    \end{itemize}
    \item $G$ admits a two-way infinite Eulerian path if and only if it satisfies the following set of conditions, called $\eulertwo$:
    \begin{itemize}
        \item The degree of each vertex is either even or infinite.
        \item $G$ has one or two ends. 
        \item If $E$ is a finite set of edges
        which induces a subgraph where all vertices have even degree, then
        $G\smallsetminus E$ has one infinite connected component.
    \end{itemize}
\end{enumerate}
\end{thm}
A recent and computable proof of this result can be found on \cite{nicanor-on-bean}. Using this characterization, it has been proven in \cite{Kuske-Lohrey} that deciding if a highly computable graph admits a one-way Eulerian path (i.e.~if it satisfies the set of conditions $\eulerone$) is $\Pi_2^0$-complete. This turns out to be the same difficulty of deciding whether a highly computable graph has at most $k$ ends, for any fixed $k > 0$. Also, it is not hard to show that deciding whether a highly computable graph admits a two-way Eulerian path (namely, it satisfies conditions $\eulertwo$) is, again, $\Pi_2^0$-complete (see \Cref{thm:eulerian-path-general} below). Moreover, we point out that in the same paper \cite{Kuske-Lohrey}, it is shown that both the above problems remain $\Pi_2^0$-complete in the case of \emph{automatic} graphs. Therefore, it is natural to ask whether the condition on the number of ends encapsulates most of the hardness of the Eulerian path problem, making this decision task easier if we consider only graphs with the correct number of ends. 
The rest of this section is devoted to confirm that indeed this is the case. 

\begin{remark}
    For simplicity and clarity, so far we have proved all our results in the setting of \emph{simple} graphs, that is, with no loops or multiple edges between vertices allowed. However, when considering Eulerian paths it is natural to allow multigraphs, since for instance the graph with a single vertex and infinitely many loops, admits both a one-way and a two-way infinite Eulerian path. For this reason, in this section we shall consider highly computable multigraphs. Let us then recall that in a multigraph $G$, the degree of the vertex $v$ equals the number of edges incident to $v$, where loops are counted twice. A multigraph $G=(V,E)$ is \emph{computable} when $V(G)$ is a decidable subset of $\N$, and we have a computable function $E(G)\colon V(G)^2\to \N$ such that $E(G)(u,v)$ is the number of edges joining $u$ to $v$ in $G$. The definitions of \emph{highly computable} multigraph, and \emph{highly computable sequence} of multigraphs are a straightforward generalization of the ones for simple graphs.  It is easy to check that all the results proved in the previous section are valid in this more general context, so we will use them without further clarifications.  
\end{remark}


\begin{remark}
We point out that the proofs of \Cref{thm:eulerian-path-general} and \Cref{thm:eulerian-path-1end} below do not introduce significantly new ideas, and are not related to the Separation Problem.  The situation is different for \Cref{thm:hardness-e_2-2fin}: here, the properties of the Separation Problem are the key to prove the upper bound in the complexity of the Eulerian Path problem when the number of ends is known to be exactly 2.
\end{remark}
The next result establishes the complexity of the problem when the number of ends can be one or two, but is not fixed.
\begin{thm}\label{thm:eulerian-path-general}
The one-way Eulerian Path Problem for highly computable graphs is $\Pi_2^0$-complete. This holds even when restricting graphs with at most two ends. The same holds for the case of the two-way Eulerian Path Problem.
\end{thm} 
\begin{proof}
    Let us begin with the upper bounds.
    Since we always deal with connected, countably infinite and locally finite graphs, we can consider the following simplified version of conditions $\eulerone$ in \Cref{egw}, which we call $\eulerone '$:
\begin{enumerate}
    \item $G$ has exactly one vertex of odd degree.
    \item $G$ has one end.
\end{enumerate}
The first condition can be rephrased as the conjunction of the $\Sigma_1^0$ statement \vir {$G$ has at least one vertex of odd degree} and the $\Pi_1^0$-statement \vir{$G$ does not have two vertices of odd degree}: therefore, we have a d.c.e.~statement. 
Moreover, we have shown that the second condition is $\Pi_2^0$.

Similarly, we can simplify conditions $\eulertwo$ in \Cref{egw} and only consider the following set of conditions, which we call $\eulertwo '$:
\begin{enumerate}
    \item Every vertex in $G$ has even degree.
    \item $G$ has either one end or two ends.
    \item If $E$ is a finite set of edges
        which induces a subgraph where all vertices have even degree, then
        $G\smallsetminus E$ has one infinite connected component.
\end{enumerate}
Observe that the first item is clearly a $\Pi_1^0$ statement, while the second statement is $\Pi_2^0$. Moreover, notice that we can write the last statement as follows: $$(\forall E \Subset E(G))(\exists n\in\N) \, ((\exists v \in G[E]) \, \deg_{G[E]}(v)\text{ is odd}) \lor (\comp_{G,n}(E)\leq 1).$$  The map $E\mapsto \comp_{G,n}(E)$ was defined in  \Cref{prop:comp-is-upper-computable}, where we also proved that it is computable. This fact shows that the last condition is also $\Pi_2^0$. 

Now we prove the lower bounds. 

To prove that deciding whether a highly computable graph admits a one-way Eulerian path is $\Pi_2^0$-complete one only needs to slightly modify the construction used in \Cref{grafo-con-ciclos}: Indeed, it is enough to perform the very same algorithm, with the only difference that now each vertex with non-negative index is joined to its successor by two edges. Thus, at any stage $s$ where $W_e[s] = W_e[s-1]$, we add edges as shown in the picture below.
    
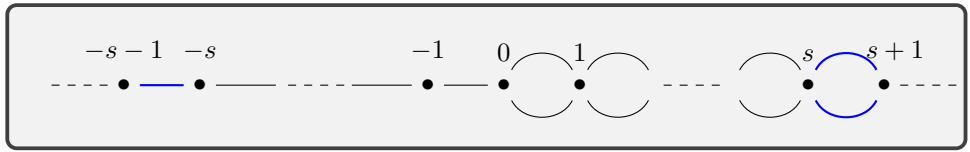
\begin{tcolorbox}
        \begin{center}
            \begin{tikzpicture}
                \node at (-1,0) (-1) {$\bullet$};
                \node at (0,0) (0) {$\bullet$};
                \node at (1,0) (1) {$\bullet$};
                \node at (-4,0) (-s) {$\bullet$};
                \node at (4,0) (s) {$\bullet$};
                \node at (-5,0) (-s-1) {$\bullet$};
                \node at (5,0) (s+1) {$\bullet$};
                \node[above = 0cm of -1] {$-1$};
                \node[above = 0cm of 0] {$0$};
                \node[above = 0cm of 1] {$1$};
                \node[above = 0cm of -s] {$-s$};
                \node[above = 0cm of s] {$s$};
                \node[above = 0cm of -s-1] {$-s-1$};
                \node[above = 0cm of s+1] {$\ \ \, s+1$};
                \draw (-1) -- (0);
                \draw (0) .. controls (0.25, 0.5) and (0.75,0.5) .. (1);
                \draw (0) .. controls (0.25, -0.5) and (0.75,-0.5) .. (1);
                \draw (-1) -- (-2,0);
                \draw[dashed] (-2.1,0) -- (-2.9,0);
                \draw (-3,0) -- (-s);
                \draw (1) .. controls (1.25, 0.5) and (1.75,0.5) .. (1.9,0.2);
                \draw (1) .. controls (1.25, -0.5) and (1.75,-0.5) .. (1.9,-0.2);
                \draw[dashed] (2.1,0) -- (2.9,0);
                \draw (3.1,0.2) .. controls (3.25,0.5) and (3.75,0.5) .. (s);
                \draw (3.1,-0.2) .. controls (3.25,-0.5) and (3.75,-0.5) .. (s);
                \draw[blue, thick] (-s) -- (-s-1);
                \draw[blue, thick] (s) .. controls (4.25, 0.5) and (4.75,0.5) .. (s+1);
                \draw[blue, thick] (s) .. controls (4.25, -0.5) and (4.75,-0.5) .. (s+1);
                \draw[dashed] (-s-1) -- (-6,0);
                \draw[dashed] (s+1) -- (6,0);
            \end{tikzpicture}
        \end{center}
    \end{tcolorbox}

    On the other hand, whenever $W_e[s] \ne W_e[s-1]$, the situation will look as in the following picture.
    
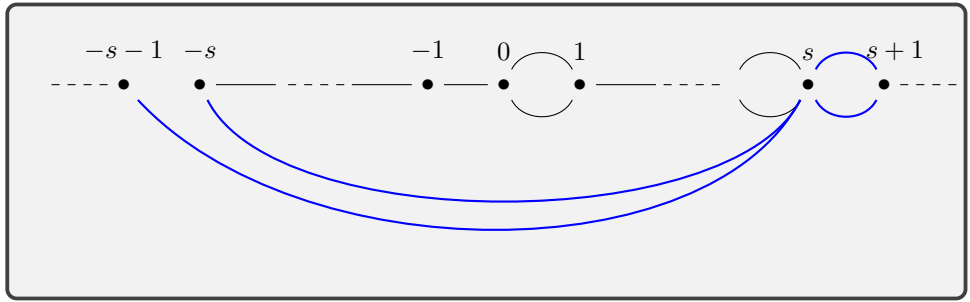
\begin{tcolorbox}
        \begin{center}
            \begin{tikzpicture}
                \node at (-1,0) (-1) {$\bullet$};
                \node at (0,0) (0) {$\bullet$};
                \node at (1,0) (1) {$\bullet$};
                \node at (-4,0) (-s) {$\bullet$};
                \node at (4,0) (s) {$\bullet$};
                \node at (-5,0) (-s-1) {$\bullet$};
                \node at (5,0) (s+1) {$\bullet$};
                \node[above = 0cm of -1] {$-1$};
                \node[above = 0cm of 0] {$0$};
                \node[above = 0cm of 1] {$1$};
                \node[above = 0cm of -s] {$-s$};
                \node[above = 0cm of s] {$s$};
                \node[above = 0cm of -s-1] {$-s-1$};
                \node[above = 0cm of s+1] {$\ \ \, s+1$};
                \draw (-1) -- (0);
                \draw (0) .. controls (0.25, 0.5) and (0.75,0.5) .. (1);
                \draw (0) .. controls (0.25, -0.5) and (0.75,-0.5) .. (1);
                \draw (-1) -- (-2,0);
                \draw[dashed] (-2.1,0) -- (-2.9,0);
                \draw (-3,0) -- (-s);
                \draw (1) -- (2,0);
                \draw[dashed] (2.1,0) -- (2.9,0);
                \draw (3.1,0.2) .. controls (3.25,0.5) and (3.75,0.5) .. (s);
                \draw (3.1,-0.2) .. controls (3.25,-0.5) and (3.75,-0.5) .. (s);
                \draw[blue, thick] (s) .. controls (4.25, 0.5) and (4.75,0.5) .. (s+1);
                \draw[blue, thick] (s) .. controls (4.25, -0.5) and (4.75,-0.5) .. (s+1);
                \draw[blue, thick] (-s) .. controls (-3,-2) and (3,-2) .. (s);
                \draw[blue, thick] (-s-1) .. controls (-2.7,-2.5) and (2.7,-2.5) .. (s);
                \draw[dashed] (-s-1) -- (-6,0);
                \draw[dashed] (s+1) -- (6,0);
            \end{tikzpicture}
        \end{center}
    \end{tcolorbox}
    
    Therefore, $G_e$ for $e \in \textbf{Fin}$ will look like this:
    
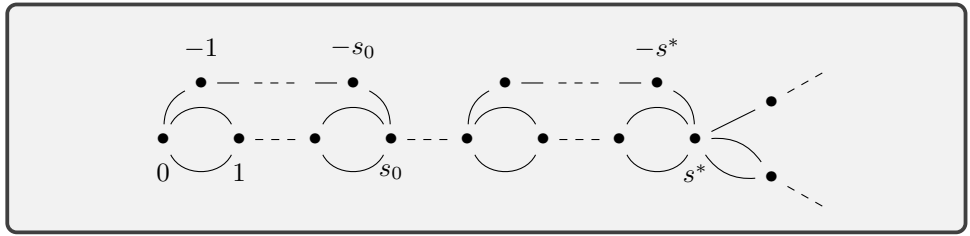
\begin{tcolorbox}
        \begin{center}
            \begin{tikzpicture}
                \node at (0,0) (0) {$\bullet$};
                \node [below = 0 cm of 0] {$0$};
                \node at (1,0) (1) {$\bullet$};
                \node [below = 0 cm of 1] {$1$};
                \draw (0) .. controls (0.25, 0.5) and (0.75,0.5) .. (1);
                \draw (0) .. controls (0.25, -0.5) and (0.75,-0.5) .. (1);
                \draw[dashed] (1) -- (1.8,0);
                \node at (2,0) (2) {$\bullet$};
                \node at (3,0) (3) {$\bullet$};
                \node [below = 0 cm of 3] {$s_0$};
                \draw (2) .. controls (2.25, 0.5) and (2.75,0.5) .. (3);
                \draw (2) .. controls (2.25, -0.5) and (2.75,-0.5) .. (3);
                \node at (0.5,0.75) (a) {$\bullet$};
                \node [above = 0 cm of a] {$-1$};
                \draw (0) to[bend left] (a);      
                \node at (2.5,.75) (b) {$\bullet$};
                \node [above = 0 cm of b] {$-s_0$};
                \draw (b) to[bend left] (3);
                \draw (a) -- (1,.75);
                \draw[dashed] (1.2,.75) -- (1.8,.75);
                \draw (2,.75) -- (b);

                \node at (4,0) (4) {$\bullet$};
                \node at (5,0) (5) {$\bullet$};
                \draw (4) .. controls (4.25, 0.5) and (4.75,0.5) .. (5);
                \draw (4) .. controls (4.25, -0.5) and (4.75,-0.5) .. (5);
                \draw[dashed] (5) -- (5.8,0);
                \node at (6,0) (6) {$\bullet$};
                \node at (7,0) (7) {$\bullet$};
                \node [below = 0 cm of 7] {$s^*$};
                \draw (6) .. controls (6.25, 0.5) and (6.75,0.5) .. (7);
                \draw (6) .. controls (6.25, -0.5) and (6.75,-0.5) .. (7);
                \node at (4.5,0.75) (c) {$\bullet$};
                \draw (4) to[bend left] (c);      
                \node at (6.5,.75) (d) {$\bullet$};
                \node [above = 0 cm of d] {$-s^*$};
                \draw (d) to[bend left] (7);
                \draw (c) -- (5,.75);
                \draw[dashed] (5.2,.75) -- (5.8,.75);
                \draw (6,.75) -- (d);

                \draw[dashed] (3) -- (4);

                \node at (8,.5) (e) {$\bullet$};
                \node at (8,-.5) (e') {$\bullet$};

                \draw (7) -- (e);
                \draw (7) to[bend left] (e');
                \draw (7) to[bend right] (e');
                \draw[dashed] (e) -- (8.7,.9);
                \draw[dashed] (e') -- (8.7,-.9);

            \end{tikzpicture}
        \end{center}
    \end{tcolorbox}
    
    While $G_e$ for $e \in \textbf{Inf}$ will be as follows:

    
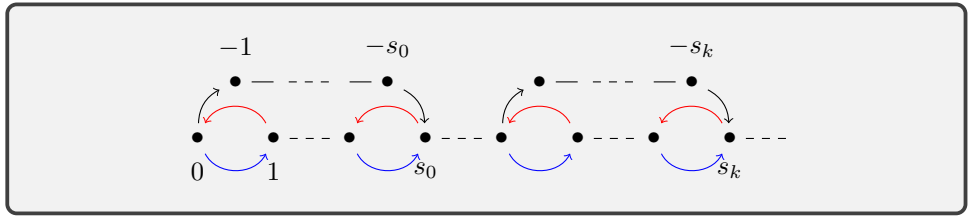
\begin{tcolorbox}
        \begin{center}
            \begin{tikzpicture}

                \node at (0,0) (0) {$\bullet$};
                \node [below = 0 cm of 0] {$0$};
                \node at (1,0) (1) {$\bullet$};
                \node [below = 0 cm of 1] {$1$};
                \draw[<-, red] (0) .. controls (0.25, 0.5) and (0.75,0.5) .. (1);
                \draw[->, blue] (0) .. controls (0.25, -0.5) and (0.75,-0.5) .. (1);
                \draw[dashed] (1) -- (1.8,0);
                \node at (2,0) (2) {$\bullet$};
                \node at (3,0) (3) {$\bullet$};
                \node [below = 0 cm of 3] {$s_0$};
                \draw[<-, red] (2) .. controls (2.25, 0.5) and (2.75,0.5) .. (3);
                \draw[->, blue] (2) .. controls (2.25, -0.5) and (2.75,-0.5) .. (3);
                \node at (0.5,0.75) (a) {$\bullet$};
                \node [above = 0 cm of a] {$-1$};
                \draw[->] (0) to[bend left] (a);      
                \node at (2.5,.75) (b) {$\bullet$};
                \node [above = 0 cm of b] {$-s_0$};
                \draw[->] (b) to[bend left] (3);
                \draw (a) -- (1,.75);
                \draw[dashed] (1.2,.75) -- (1.8,.75);
                \draw (2,.75) -- (b);

                \node at (4,0) (4) {$\bullet$};
                \node at (5,0) (5) {$\bullet$};
                \draw[<-, red] (4) .. controls (4.25, 0.5) and (4.75,0.5) .. (5);
                \draw[->, blue] (4) .. controls (4.25, -0.5) and (4.75,-0.5) .. (5);
                \draw[dashed] (5) -- (5.8,0);
                \node at (6,0) (6) {$\bullet$};
                \node at (7,0) (7) {$\bullet$};
                \node [below = 0 cm of 7] {$s_k$};
                \draw[<-, red] (6) .. controls (6.25, 0.5) and (6.75,0.5) .. (7);
                \draw[->, blue] (6) .. controls (6.25, -0.5) and (6.75,-0.5) .. (7);
                \node at (4.5,0.75) (c) {$\bullet$};
                \draw[->] (4) to[bend left] (c);      
                \node at (6.5,.75) (d) {$\bullet$};
                \node [above = 0 cm of d] {$-s_k$};
                \draw[->] (d) to[bend left] (7);
                \draw (c) -- (5,.75);
                \draw[dashed] (5.2,.75) -- (5.8,.75);
                \draw (6,.75) -- (d);

                \draw[dashed] (3) -- (4);

                \draw[dashed] (7) -- (7.8,0);

            \end{tikzpicture}
        \end{center}
    \end{tcolorbox}

    Observe that, for every $e$, the only vertex with odd degree in $G_e$ is $0$: then $G_e$ admits an infinite Eulerian path if and only if $G_e$ has one end, which happens if and only $e \in \textbf{Inf}$.

    We now turn to the two-way Eulerian Path Problem. To prove its $\Pi_2^0$-completeness, again, it is enough to slightly modify the construction in \Cref{grafo-con-ciclos}: indeed, we consider graphs $H_e$ which are built in the very same way, except that every edge is now duplicated.  We hence obtain a uniformly highly computable sequence of multigraphs $(H_e)_{e \in \N}$ where each $H_e$ has either one or two ends and where each vertex has even degree. We claim that 
    $$\{e: \, H_e \text{ admits a two-way Eulerian path}\} = \{e: \, H_e \text{ has one end}\}.$$
    Notice that this is enough to conclude the proof, as we have seen the set on the right-hand side is $\Pi_2^0$-complete.

    Let us first consider the case where $H_e$ has one end. As every vertex in $H_e$ has even degree, it follows that it verifies $\eulertwo$, so it admits a two-way infinite Eulerian path. For concreteness, let us describe how to obtain such a path. For this we let $s_0=0$, and denote by $s_k$ the $k$-th step in which a new element is enumerated into $W_e$. As we are in the case where $W_e$ is infinite, we have a well-defined infinite sequence $(s_k)_{k \in \N}$. Then each subgraph $H_e[\{v\in V(H_e): \, s_k \le |v| \le s_{k+1}\} \smallsetminus \{-s_k\}$ is even and connected, so by Euler's theorem admits an Eulerian cycle. We can split such a cycle into two disjoint paths from $s_k$ to $s_{k+1}$, shown in blue and in red in the picture below. These paths can be joined to construct the desired two-way Eulerian path.
    \begin{tcolorbox}
        \begin{center}
            \begin{tikzpicture}
                \node at (0,0) (0) {$\bullet$};
                \node at (45:1cm) (1) {$\bullet$};
                \node at (-45:1cm) (-1) {$\bullet$};
                \node at ($(1)+(1,0)$) (2) {$\bullet$};
                \node at ($(-1)+(1,0)$) (-2) {$\bullet$};
                \node at ($(2)+(1,0)$) (3) {$\bullet$};
                \node at ($(-2)+(1,0)$) (-3) {$\bullet$};
                \node at ($(3)+(1,0)$) (4) {$\bullet$};
                \node at ($(-3)+(1,0)$) (-4) {$\bullet$};
                \node at ($(4)+(-45:1cm)$) (5) {$\bullet$};
                \draw[red, thick, ->] (0) to[bend left] (1);
                \draw[red, thick, ->] (1) to[bend left] (2);
                \draw[dashed] (2) -- (3);
                \draw[red, thick, ->] (3) to[bend left] (4);
                \draw[red, thick, ->] (4) to[bend left] (5);
                \draw[blue, thick, ->] (0) to[bend right] (1);
                \draw[blue, thick, ->] (1) to[bend right] (2);
                \draw[blue, thick, ->] (3) to[bend right] (4);
                \draw[blue, thick, ->] (4) to[bend right] (5);
                \draw[blue, thick, ->] (0) to[bend right] (-1);
                \draw[blue, thick, ->] (-1) to[bend right] (-2);
                \draw[dashed] (-2) -- (-3);
                \draw[blue, thick, ->] (-3) to[bend right] (-4);
                \draw[blue, thick, ->] (-4) to[bend right] (5);
                \draw[blue, thick, <-] (0) to[bend left] (-1);
                \draw[blue, thick, <-] (-1) to[bend left] (-2);
                \draw[blue, thick, <-] (-3) to[bend left] (-4);
                \draw[blue, thick, <-] (-4) to[bend left] (5);
                \node[above left = -0.3 cm of 0] {$s_k$};
                \node[above left = -0.3 cm of 1] {$s_k+1$};
                \node[below left = -0.3 cm of -1] {$-(s_k+1)$};
                \node[above right = -0.3 cm of 4] {$s_{k+1}-1$};
                \node[below right = -0.3 cm of -4] {$-s_{k+1}$};
                \node[above right = -0.3 cm of 5] {$s_{k+1}$};

                \node at ($(5)+(3,0)$) (0') {$\bullet$};
                \node at ($(0')+(1,0)$) (1') {$\bullet$};
                \draw[red, thick, ->] (0') .. controls ($(0')+(0.25, 0.5)$) and ($(0')+(0.75,0.5)$) .. (1');
                \draw[blue, thick, ->] (0') to[bend left] (1');
                \draw[blue, thick, <-] (0') to[bend right] (1');
                \draw[blue, thick, ->] (0') .. controls ($(0')+(0.25,-0.5)$) and ($(0')+(0.75,-0.5)$) .. (1');
                \node at ($(1')+(1,0)$) (2') {$\bullet$};
                \draw[red, thick, ->] (1') .. controls ($(1')+(0.25, 0.5)$) and ($(1')+(0.75,0.5)$) .. (2');
                \draw[blue, thick, ->] (1') to[bend left] (2');
                \draw[blue, thick, <-] (1') to[bend right] (2');
                \draw[blue, thick, ->] (1') .. controls ($(1')+(0.25,-0.5)$) and ($(1')+(0.75,-0.5)$) .. (2');
                
                \draw[dashed] (2') -- ++(0:0.75cm);
                \node[above = 0 cm of 0'] {$0$};
                \node[above = 0 cm of 1'] {$s_1$};
                \node[above = 0 cm of 2'] {$s_2$};
            \end{tikzpicture}
        \end{center}
    \end{tcolorbox}
    On the other hand, if $H_e$ has two ends, then, by construction, $W_e$ is finite. So, there must be a stage $s_k$ in which some element enters $W_e$ for the last time. Let $E\Subset E(H_e)$ be the set whose elements are the two edges joining $s_k$ to $s_{k+1}$. Then $H_e[E]$ is an even graph, while $H_e\smallsetminus E$ has two infinite connected components, so $H_e$ does not satisfy $\eulertwo$.
\end{proof}

For graphs with exactly one end, we have the following result. 
 \begin{thm} \label{thm:eulerian-path-1end}  Let $\mathcal{G}_{one}$ denote the class of connected one-ended  graphs. Then the following holds:
 \begin{enumerate}
       \item The one-way Eulerian Path Problem over highly computable graphs in $\mathcal{G}_{one}$ is d.c.e.-complete.
        \item The two-way Eulerian Path Problem over highly computable graphs in $\mathcal{G}_{one}$ is $\Pi_1^0$-complete. 
        \item Both the one-way and the two-way Eulerian Path Problems over automatic graphs in $\mathcal{G}_{one}$ are decidable. 
    \end{enumerate}
\end{thm}

\begin{proof}
    We first prove item 1. For the upper bound, it is sufficient to notice that, in case we are promised that the given graph has one end, we only need to check the first condition of $\eulerone'$, which is a d.c.e.~statement.
    
    Now we turn to the lower bound. Let $W$ be a d.c.e.-complete set (the existence of such set has been proved in \cite{Ershov}) and $f: \N^2 \to \N$ be a computable approximation of $W$: since $W$ is d.c.e., we can assume that $f$ changes mind on each element at most twice. Moreover, we can assume that $f(e,0)=0$ for every $e$. 

For every $e$, we define $G_e$ by letting $V(G_e) = \N$. Moreover, we define the edge set of $E(G_e)$ in stages. 
    
    At every stage $s$, we look at the value of $f(e,s)$. While $f(e,s)=0$, we keep adding 2 edges between $s$ and $s+1$. 
    
            
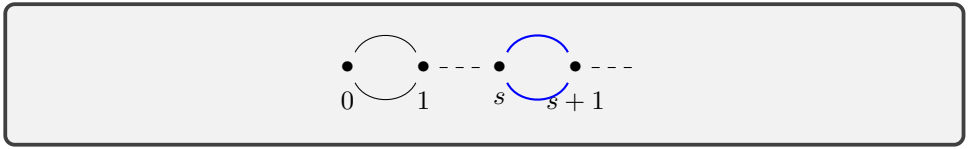
\begin{tcolorbox}
        \begin{center}
            \begin{tikzpicture}
                \node at (0,0) (0) {$\bullet$};
                \node [below = 0 cm of 0] {$0$};
                \node at (1,0) (1) {$\bullet$};
                \node [below = 0 cm of 1] {$1$};
                \draw (0) .. controls (0.25, 0.5) and (0.75,0.5) .. (1);
                \draw (0) .. controls (0.25, -0.5) and (0.75,-0.5) .. (1);
                \draw[dashed] (1) -- (1.8,0);
                \node at (2,0) (2) {$\bullet$};
                \node [below = 0 cm of 2] {$s$};
                \node at (3,0) (3) {$\bullet$};
                \node [below = 0 cm of 3] {$s+1$};
                \draw[thick, blue] (2) .. controls (2.25, 0.5) and (2.75,0.5) .. (3);
                \draw[thick, blue] (2) .. controls (2.25, -0.5) and (2.75,-0.5) .. (3);
                \draw[dashed] (3) -- (3.8,0);
            \end{tikzpicture}
        \end{center}
    \end{tcolorbox}
    
    But whenever $f(e,s)=1$, we instead start adding only one edge between $s$ and $s+1$. 
    
        
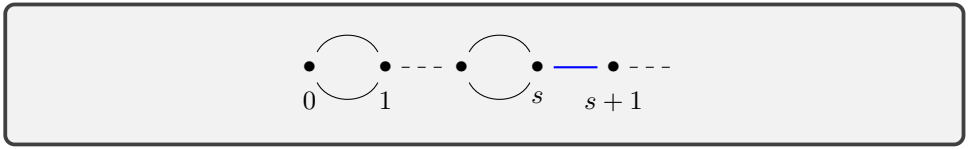
\begin{tcolorbox}
        \begin{center}
            \begin{tikzpicture}
                \node at (0,0) (0) {$\bullet$};
                \node [below = 0 cm of 0] {$0$};
                \node at (1,0) (1) {$\bullet$};
                \node [below = 0 cm of 1] {$1$};
                \draw (0) .. controls (0.25, 0.5) and (0.75,0.5) .. (1);
                \draw (0) .. controls (0.25, -0.5) and (0.75,-0.5) .. (1);
                \draw[dashed] (1) -- (1.8,0);
                \node at (2,0) (2) {$\bullet$};
                \node at (3,0) (3) {$\bullet$};
                \draw (2) .. controls (2.25, 0.5) and (2.75,0.5) .. (3);
                \draw (2) .. controls (2.25, -0.5) and (2.75,-0.5) .. (3);
                \node at (4,0) (4) {$\bullet$};
                \node [below = 0 cm of 3] {$s$};
                \node [below = 0 cm of 4] {$s+1$};
                \draw[thick, blue] (3) -- (4);
                \draw[dashed] (4) -- (4.8,0);
            \end{tikzpicture}
        \end{center}
    \end{tcolorbox}
    
    Finally, it might also happen that $f$ changes its mind again at stage $s$: hence, we go back to adding  2 edges between $s$ and $s+1$.
    
        
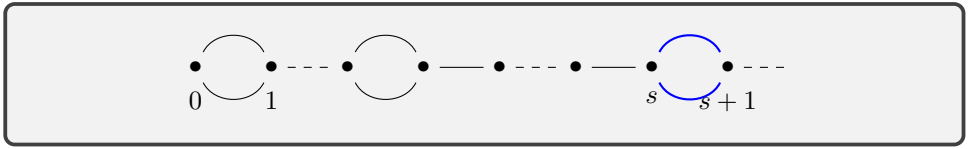
\begin{tcolorbox}
        \begin{center}
            \begin{tikzpicture}
                \node at (0,0) (0) {$\bullet$};
                \node [below = 0 cm of 0] {$0$};
                \node at (1,0) (1) {$\bullet$};
                \node [below = 0 cm of 1] {$1$};
                \draw (0) .. controls (0.25, 0.5) and (0.75,0.5) .. (1);
                \draw (0) .. controls (0.25, -0.5) and (0.75,-0.5) .. (1);
                \draw[dashed] (1) -- (1.8,0);
                \node at (2,0) (2) {$\bullet$};
                \node at (3,0) (3) {$\bullet$};
                \draw (2) .. controls (2.25, 0.5) and (2.75,0.5) .. (3);
                \draw (2) .. controls (2.25, -0.5) and (2.75,-0.5) .. (3);
                \node at (4,0) (4) {$\bullet$};
                \draw(3) -- (4);
                \draw[dashed] (4) -- (4.8,0);
                \node at (5,0) (5) {$\bullet$};
                \node at (6,0) (6) {$\bullet$};
                \node [below = 0 cm of 6] {$s$};
                \draw (5) -- (6);
                \node at (7,0) (7) {$\bullet$};
                \node [below = 0 cm of 7] {$s+1$};
                \draw[thick, blue] (6) .. controls (6.25, 0.5) and (6.75,0.5) .. (7);
                \draw[thick, blue] (6) .. controls (6.25, -0.5) and (6.75,-0.5) .. (7);
                \draw[dashed] (7) -- (7.8,0);
            \end{tikzpicture}
        \end{center}
    \end{tcolorbox}

    Notice that all $G_e$'s constructed in this way are highly computable and have one end.
    
    Moreover, if $f$ never changes its mind on $e$, then all vertices in $G_e$ have even degree, hence $G_e$ cannot have a one-way infinite Eulerian path. On the other hand, if $f$ changes its mind on $e$ at step $s$ and never changes its mind again, we have that $s$ has odd degree, while all other vertices have even degree. In this case, which is exactly when $e \in W$, $G_e$ admits a one-way infinite Eulerian path.
    
       For example, we have the path described in the picture below: such a path starts at the vertex pointed by the double arrow, crosses our graph \vir{backwards} going through the red edges until it reaches the vertex $0$ and then proceeds forever visiting each black edge in order.

    
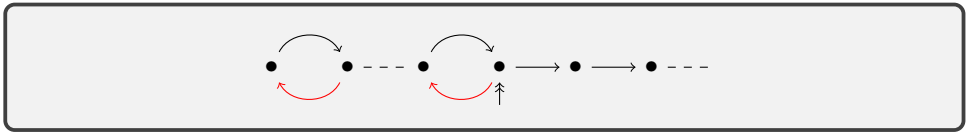
\begin{tcolorbox}
        \begin{center}

            \begin{tikzpicture}
                \node at (0,0) (0) {$\bullet$};
                \node at (1,0) (1) {$\bullet$};
                \draw[->] (0) .. controls (0.25, 0.5) and (0.75,0.5) .. (1);
                \draw[<-, red] (0) .. controls (0.25, -0.5) and (0.75,-0.5) .. (1);
                \draw[dashed] (1) -- (1.8,0);
                \node at (2,0) (2) {$\bullet$};
                \node at (3,0) (3) {$\bullet$};
                \draw[->] (2) .. controls (2.25, 0.5) and (2.75,0.5) .. (3);
                \draw[<-, red] (2) .. controls (2.25, -0.5) and (2.75,-0.5) .. (3);
                \node at (4,0) (4) {$\bullet$};
                \draw[->] (3) -- (4);
                \node at (5,0) (5) {$\bullet$};
                \draw[->] (4) -- (5);
                \draw[dashed] (5) -- (5.8,0);
                \draw[->>] (3,-.5) -- (3);
            \end{tikzpicture}
        \end{center}
    \end{tcolorbox} 
    
    Finally, if $f$ changes its mind twice on $e$, say at stages $s$ and $s'$, both vertices $s$ and $s'$ have odd degree: hence $G_e$ does not admit Eulerian paths.

Now let's prove item 2. The upper bound is again straightforward: when we deal only with graphs with one end, we only need to check the first condition of $\eulertwo '$, which is clearly a $\Pi_1^0$ statement.

To prove the $\Pi_1^0$-completeness, let us fix a computable enumeration $(\varphi_{e})_{e \in \N}$ and define the highly computable sequence of multigraphs $(G_e)_{e \in \N}$ as follows. For every $e$, we let $V(G_e) = \N$. Then, we always add two edges between $s$ and $s+1$, unless $\varphi_e$ halts in exactly $s$ steps, in which case we only add one edge. Thus, we have the following two cases:

    
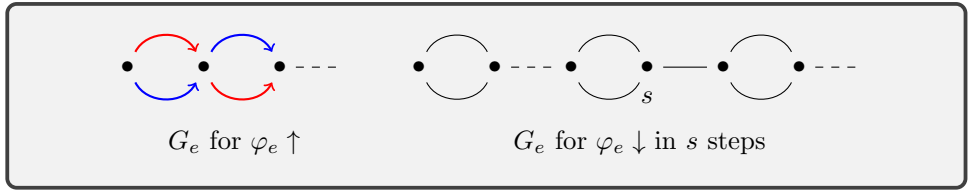
\begin{tcolorbox}
        \begin{center}

            \begin{tikzpicture}
                \node at (0,0) (0) {$\bullet$};
                \node at (1,0) (1) {$\bullet$};
                \draw[red, thick, ->] (0) .. controls (0.25, 0.5) and (0.75,0.5) .. (1);
                \draw[blue, thick, ->] (0) .. controls (0.25, -0.5) and (0.75,-0.5) .. (1);
                \node at (2,0) (2) {$\bullet$};
                \draw[blue, thick, ->] (1) .. controls (1.25, 0.5) and (1.75,0.5) .. (2);
                \draw[red, thick, ->] (1) .. controls (1.25, -0.5) and (1.75,-0.5) .. (2);
                
                \draw[dashed] (2) -- (2.8,0);
                \node at (1.4, -1) {$G_e$ for $\varphi_e \uparrow$};
            \end{tikzpicture}
            \qquad
             \begin{tikzpicture}
                \node at (0,0) (0) {$\bullet$};
                \node at (1,0) (1) {$\bullet$};
                \draw (0) .. controls (0.25, 0.5) and (0.75,0.5) .. (1);
                \draw (0) .. controls (0.25, -0.5) and (0.75,-0.5) .. (1);

                 \draw[dashed] (1) -- (1.8,0);
                \node at (2,0) (2) {$\bullet$};
                \node at (3,0) (3) {$\bullet$};
                
                \draw (2) .. controls (2.25, 0.5) and (2.75,0.5) .. (3);
                \draw (2) .. controls (2.25, -0.5) and (2.75,-0.5) .. (3);
                \node at (4,0) (4) {$\bullet$};
                \draw (3) -- (4);
                \node at (5,0) (5) {$\bullet$};
                \draw (4) .. controls (4.25, 0.5) and (4.75,0.5) .. (5);
                \draw (4) .. controls (4.25, -0.5) and (4.75,-0.5) .. (5);
                \node [below = 0 cm of 3] {$s$};

                \draw[dashed] (5) -- (5.8,0);
                \node at (2.9, -1) {$G_e$ for $\varphi_e \downarrow$ in $s$ steps};
            \end{tikzpicture}
        \end{center}
    \end{tcolorbox}  
    It is clear that the vertices $s$ and $s+1$ in $G_e$ have odd degree if and only if $\varphi_{e}$ halts in exactly $s$ steps. It follows that $G_e$ admits a two-way Eulerian path if and only if $e \in \overline{\HP}$.

Finally, we turn to the automatic case (item 3).

Automatic graphs have a decidable first order theory. Even more, the sets that can be defined by first order formulas are regular with respect to an automatic presentation  \cite{khoussainov_automatic_1995}. This proves the decidability of the first order theory, as the emptiness problem for regular languages is decidable. Several works show that we can enrich the first order theory with  counting quantifiers such as  ``there are $k$ mod $n$'', and ``there are infinitely many'', and that the sets defined with these formulas are still regular. As a direct consequence of \cite[Propositions 3.5 and 3.6]{kuske_where_2011}, we get the following result.

\begin{lem} \label{lem:automatic}
    There is an algorithm which, given an automatic graph $G$, an automatic presentation for $(V(G),R(G))$, and a sentence written with first order quantifiers plus quantifiers of the form $\exists^{\text{even}}$ (there is an even number), $\exists^{\text{odd}}$ (there is an odd number), and $\exists^{\infty}$ (there is an infinite number),  returns the truth value of the sentence. 
\end{lem}

Given this premise, item 3 follows easily. 
Indeed, by \Cref{egw}, we know that a graph $G$ admits a one-way infinite Eulerian path if and only if either $G$ has exactly one vertex whose degree is odd, or $G$ has at least one vertex with infinite degree and no vertex with odd degree. Similarly, $G$ admits a two-way infinite Eulerian path if and only if all vertices have either even or infinite degree. Both these conditions can be written as sentences using only the kind of quantifiers allowed in \Cref{lem:automatic}, so they are decidable given the automatic presentation of $G$, in case $G$ is automatic. 
\end{proof}
We now turn to the case of graphs having exactly two ends. Interestingly, we find that in this case the existence of a two-way Eulerian path exhausts exactly the $m$-degrees of $\Delta_2^0$ sets. We note that, in order to prove the $\Delta_2^0$ upper bound, we strongly rely on the properties of the Separation Problem. 
\begin{thm} \label{thm:hardness-e_2-2fin} Let $\mathcal{G}_{two}$
denote the class of connected two-ended  highly computable graphs. Then the two-way Eulerian Path Problem over $\mathcal{G}_{two}$ is $\Delta_2^0$. Moreover, for every $\Delta_2^0$ set $X\subset \N$ there is a uniformly highly computable family of graphs $(G_e)_{e\in\N}$ with two ends, such that 
    $$X=\{e\in \N \mid G_e \ \text{admits a two-way infinite Eulerian path}\}.$$ 
\end{thm}
\begin{proof}
    We start by proving the upper bound. Let $(G_e)_{e \in \N}$ be a uniformly highly computable family of connected graphs with two ends. Due to the above restrictions on the graphs of such family, conditions $\eulertwo$ from \Cref{egw} are greatly simplified, and we are left with verifying the following set of conditions, which we call $\eulertwo^{\star}$:
    \begin{enumerate}
        \item $(\forall v \in V(G)) \ \deg_G(v) \ \text{is even}$.
        \item $(\forall E \Subset E(G)) \ (\forall v \in V(G[E]) \, \deg_{G[E]}(v) \ \text{is even}) \implies \comp_G(E) = 1$.
    \end{enumerate}
     A Turing reduction which decides whether $G_e$ satisfies $\eulertwo^{\star}$ using $\HP$ is the following. We start by verifying that every vertex in $G_e$ has even degree. This is a $\Pi_1^0$ condition and, hence, decidable with oracle $\HP$. Now our task is to verify that $G_e \smallsetminus E$ has a single infinite connected component, for every set of edges $E$ that induces an even subgraph of $G_e$. For this purpose we compute a set $S\in\sep{G_e}$, which is possible with oracle $\HP$ by \Cref{complexity-of-sep-uniform-version}. Now let $\varphi$ be the function from \Cref{thm:magic_function}. By the Recursion Theorem, from the index $e$ we can compute an index $e'$ of a Turing machine which runs $\varphi(G_e,S,2,E)$ for every finite set of edges $E$ that induces an even subgraph of $G_e$, and that halts whenever it finds some $E$ such that $\varphi(G_e,S,2,E)\geq 2$. Then $\varphi_{e'}$ halts if and only if $G_e$ has a finite set of edges $E$ for which $G_e \smallsetminus E$ has more than one infinite connected component, namely if and only if $G_e$ does not satisfy $\eulertwo^{\star}$. This proves that, in order to check whether $G_e$ satisfies $\eulertwo^{\star}$, it suffices to ask whether $e'\in \HP$. This concludes the desired Turing reduction.

     We now sketch the proof of the fact that, given any $\Delta_2^0$ set $X$, there is a uniformly highly computable sequence of graphs $(G_e)_{e \in \N}$ with 
     \[\{e: \ G_e \ \text{admits a two-way Eulerian path} \, \}=X.\] The construction is similar in spirit to the one used in \Cref{grafo-con-ciclos}. Let $x(e,s)$ be a computable approximation to the characteristic sequence of $X$. Without loss of generality, we can assume that $x(e,-1)=x(e,0)=0$ for every $e$. Now, for every $e$, we let $V(G_e) = \Z$, while the set of edges is defined by the following effective procedure:
     \begin{enumerate}
    \setcounter{enumi}{-1}
        \item We begin by adding two edges between $0, 1$ and $0, -1$. 
        \item We keep adding two edges between both $s, s+1$ and $-s, -(s+1)$ until we see that $x(e,s-1)=x(e,s)=0$ (that is, our computable approximation believes that $e \notin X$), so that the graphs looks like a \vir{chain}, namely a line where two adjacent vertices are joined by two edges.
    \end{enumerate}
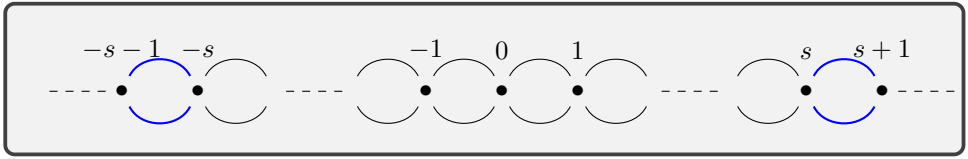
    \begin{tcolorbox}
        \begin{center}
            \begin{tikzpicture}
                \node at (-1,0) (-1) {$\bullet$};
                \node at (0,0) (0) {$\bullet$};
                \node at (1,0) (1) {$\bullet$};
                \node at (-4,0) (-s) {$\bullet$};
                \node at (4,0) (s) {$\bullet$};
                \node at (-5,0) (-s-1) {$\bullet$};
                \node at (5,0) (s+1) {$\bullet$};
                \node[above =0.1cm of -1] {$-1$};
                \node[above =0.1cm of 0] {$0$};
                \node[above =0.1cm of 1]{$1$};
                \node[above =0.1cm of -s] {$-s$};
                \node[above =0.1 cm of s]{$s$};
                \node[above =0.1cm of -s-1] {$-s-1$};
                \node[above =0.1cm of s+1] {$s+1$};
                \draw (-1) .. controls (-0.75, 0.5) and (-0.25,0.5) .. (0);
                \draw (-1) .. controls (-0.75, -0.5) and (-0.25,-0.5) .. (0);
                \draw (0) .. controls (0.25, 0.5) and (0.75,0.5) .. (1);
                \draw (0) .. controls (0.25, -0.5) and (0.75,-0.5) .. (1);
                \draw (-1) .. controls (-1.25, 0.5) and (-1.75,0.5) .. (-1.9,0.2);
                \draw (-1) .. controls (-1.25, -0.5) and (-1.75,-0.5) .. (-1.9,-0.2);
                \draw[dashed] (-2.1,0) -- (-2.9,0);
                \draw (-3.1,0.2) .. controls (-3.25, 0.5) and (-3.75,0.5) .. (-s);
                \draw (-3.1,-0.2) .. controls (-3.25, -0.5) and (-3.75,-0.5) .. (-s);
                \draw (1) .. controls (1.25, 0.5) and (1.75,0.5) .. (1.9,0.2);
                \draw (1) .. controls (1.25, -0.5) and (1.75,-0.5) .. (1.9,-0.2);
                \draw[dashed] (2.1,0) -- (2.9,0);
                \draw (3.1,0.2) .. controls (3.25,0.5) and (3.75,0.5) .. (s);
                \draw (3.1,-0.2) .. controls (3.25,-0.5) and (3.75,-0.5) .. (s);
                \draw[blue, thick] (-s) .. controls (-4.25, 0.5) and (-4.75,0.5) .. (-s-1);
                \draw[blue, thick] (-s) .. controls (-4.25, -0.5) and (-4.75,-0.5) .. (-s-1);
                \draw[blue, thick] (s) .. controls (4.25, 0.5) and (4.75,0.5) .. (s+1);
                \draw[blue, thick] (s) .. controls (4.25, -0.5) and (4.75,-0.5) .. (s+1);
                \draw[dashed] (-s-1) -- (-6,0);
                \draw[dashed] (s+1) -- (6,0);
            \end{tikzpicture}
        \end{center}
    \end{tcolorbox}
    \begin{enumerate}
    \setcounter{enumi}{1}
        \item If at stage $s$ we see that $x(e,s-1)=0$ and $x(e,s)=1$ (namely, our approximation changes its mind and puts $e$ into $X$), we add two edges between $-s$ and $s$, then we add one edge between both $s,s+1$ and $s, -(s+1)$, while putting no edge between $-s$ and $-(s+1)$. In this way, we create a cycle graph with vertex set $\{-s, \dots, s\}$, where each pair of adjacent vertices is joined by two edges.
    \end{enumerate}
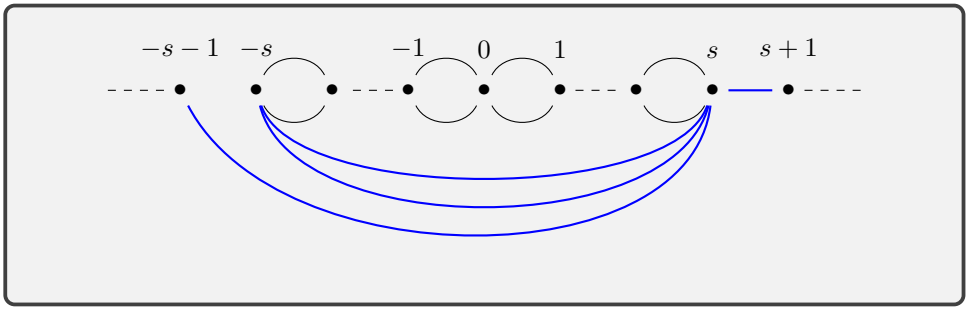
    \begin{tcolorbox}
        \begin{center}
            \begin{tikzpicture}
                \node at (0,0) (s) {$\bullet$};
                \node at (.75,0.5) (s-1) {$\bullet$};
                \node at (-.75,0.5) (-s) {$\bullet$};
                \draw (s) to[bend left] (s-1); 
                \draw (s) to[bend right] (s-1); 
                \draw[blue, thick] (s) to[bend left] (-s); 
                \draw[blue, thick] (s) to[bend right] (-s); 
                \node at (0.75,1.25) (1) {$\bullet$};
                \node at (-0.75,1.25) (-1) {$\bullet$};
                \node at (0,1.75) (0) {$\bullet$};
                \draw (-1) to[bend left] (0); 
                \draw (-1) to[bend right] (0);
                \draw (1) to[bend left] (0); 
                \draw (1) to[bend right] (0); 
                \draw[dashed] (s-1) to[bend right] (1);
                \draw[dashed] (-s) to[bend left] (-1);
                \node at (1,0) (s+1) {$\bullet$};
                \node at (-1,0) (-s-1) {$\bullet$};
                \draw[blue, thick] (s) -- (s+1);
                \draw[blue, thick] (s) -- (-s-1);
                \node[below =0cm of s] {$s$};
                \node[below =0cm of s+1] {$s+1$};
                \node[below =0cm of -s-1] {$-s-1$};
                \node[left=0cm of -s] {$-s$};
                \node[right=0cm of s-1] {$s-1$};
                \node[above =0cm of 0] {$0$};
                \node[right=0cm of 1] {$1$};
                \node[left=0cm of -1] {$-1$};
                \draw[dashed] (-s-1) -- (-2,0);
                \draw[dashed] (s+1) -- (2,0);
            \end{tikzpicture}
        \end{center}
    \end{tcolorbox}
    \begin{enumerate}
    \setcounter{enumi}{2}
        \item We keep adding one edge between both $s, s+1$ and $-s, -(s+1)$ until we see that $x(e,s-1)=x(e,s)=1$ (that is, our computable approximation believes that $e \in X$), forming a line.
    \end{enumerate}
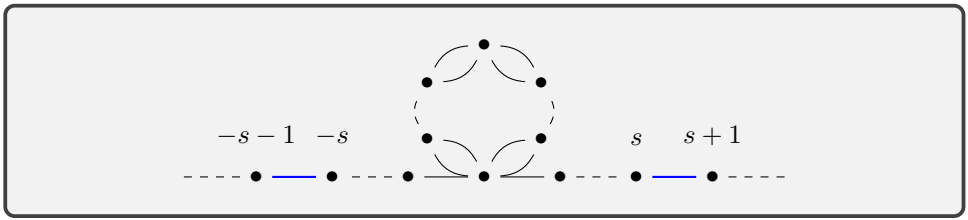
    \begin{tcolorbox}
        \begin{center}
            \begin{tikzpicture}
                \node at (0,0) (0) {$\bullet$};
                \node at (.75,0.5) (0') {$\bullet$};
                \node at (-.75,0.5) (0'') {$\bullet$};
                \draw (0) to[bend left] (0'); 
                \draw (0) to[bend right] (0'); 
                \draw (0) to[bend left] (0''); 
                \draw (0) to[bend right] (0''); 
                \node at (0.75,1.25) (1') {$\bullet$};
                \node at (-0.75,1.25) (1'') {$\bullet$};
                \node at (0,1.75) (1''') {$\bullet$};
                \draw (1') to[bend left] (1'''); 
                \draw (1') to[bend right] (1''');
                \draw (1'') to[bend left] (1'''); 
                \draw (1'') to[bend right] (1'''); 
                \draw[dashed] (0') to[bend right] (1');
                \draw[dashed] (0'') to[bend left] (1'');
                \node at (1,0) (1) {$\bullet$};
                \node at (-1,0) (-1) {$\bullet$};
                \draw (0) -- (1);
                \draw (0) -- (-1);
                \node at (2,0) (s) {$\bullet$};
                \node[above =0.1 cm of s] {$s$};
                \node at (-2,0) (-s) {$\bullet$};
                \node[above =0.1 cm  of -s] {$-s$};
                \draw[dashed] (1) -- (s);
                \draw[dashed] (-1) -- (-s);
                \node at (3,0) (s+1) {$\bullet$};
                \node[above =0.1 cm of s+1] {$s+1$};
                \node at (-3,0) (-s-1) {$\bullet$};
                \node[above =0.1 cm of -s-1] {$-s-1$};
                \draw[blue, thick] (s) -- (s+1);
                \draw[blue, thick] (-s) -- (-s-1);
                \draw[dashed] (s+1) -- (4,0);
                \draw[dashed] (-s-1) -- (-4,0);
            \end{tikzpicture}
        \end{center}
    \end{tcolorbox}
    \begin{enumerate}
    \setcounter{enumi}{3}
        \item If at stage $s$ we see that $x(e,s-1)=1$ and $x(e,s)=0$ (namely, our approximation changes its mind again and removes $e$ from $X$), we add one edge between $-s$ and $s$, then we add two edges between both $s,s+1$ and $s, -(s+1)$, while putting no edge between $-s$ and $-(s+1)$. So, we form a new cycle, this time with only one edge joining two adjacent vertices. 
    \end{enumerate}
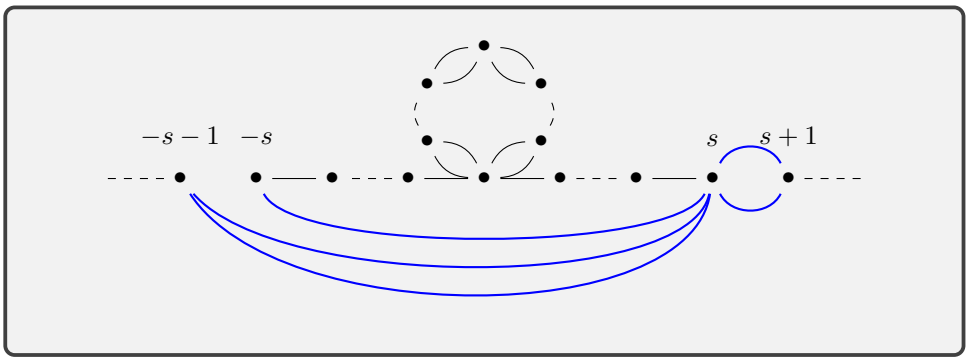
    \begin{tcolorbox}
        \begin{center}
\begin{tikzpicture}[rotate=90]
                \node at (0,0) (0) {$\bullet$};
                \node[right=0cm of 0] {$s$};
                \node at (.75,0.5) (0') {$\bullet$};
                \node[above =0cm of 0'] {$s-1$};
                \node at (-.75,0.5) (0'') {$\bullet$}; 
                \node[below =0cm of 0''] {$-s$};
                \draw (0) to[bend right] (0'); 
                \draw[blue, thick] (0) to[bend left] (0''); 
                \node at (0.75,1.25) (0''') {$\bullet$};
                \node at (-0.75,1.25) (0'''') {$\bullet$};
                \draw[dashed] (0') to[bend right] (0''');
                \draw[dashed] (0'') to[bend left] (0'''');
                \node at (0,1.75) (0''''') {$\bullet$};
                \draw (0''') to[bend right] (0''''');
                \draw (0'''') to[bend left] (0''''');
                \node at (.75,2.25) (1') {$\bullet$};
                \node at (-.75,2.25) (1'') {$\bullet$};
                \draw (0''''') to[bend left] (1'); 
                \draw (0''''') to[bend right] (1'); 
                \draw (0''''') to[bend left] (1''); 
                \draw (0''''') to[bend right] (1'');
                \node at (.75,3) (1''') {$\bullet$};
                \node at (-.75,3) (1'''') {$\bullet$};
                \node at (0,3.5) (1''''') {$\bullet$};
                \draw[dashed] (1') to[bend right] (1''');
                \draw[dashed] (1'') to[bend left] (1'''');
                \draw (1''') to[bend right] (1''''');
                \draw (1''') to[bend left] (1''''');
                \draw (1'''') to[bend right] (1''''');
                \draw (1'''') to[bend left] (1''''');
                \node at (-.75,-.5) (-1) {$\bullet$};
                \node[right=0cm of -1] {$-s-1$};
                \node at (.75,-.5) (1) {$\bullet$};
                \node[right=0cm of 1] {$s+1$};
                \draw[blue, thick] (0) to[bend right] (1);
                \draw[blue, thick] (0) to[bend left] (1);
                \draw[blue, thick] (0) to[bend right] (-1);
                \draw[blue, thick] (0) to[bend left] (-1);
                \draw[dashed] (1) -- (1.5,-1);
                \draw[dashed] (-1) -- (-1.5,-1);
            \end{tikzpicture}
        \end{center}
    \end{tcolorbox}
    \begin{enumerate}
    \setcounter{enumi}{4}
        \item We restart from item (1).
    \end{enumerate}

    Notice that the graph $G_e$ constructed in such a way always has two ends. Moreover, it is easy to check that $G_e$ satisfies $\eulertwo^{\star}$ if and only if $e \in X$.

     Let $H_1$ be the graph with vertex set $\Z$ and where each pair of two consecutive vertices $n, n+1$ is joined by a single edge, and let $H_2$ be the graph with vertex set $\Z$ and with each two consecutives vertices $n, n+1$ are joined by exactly two edges: notice that $H_1$ admits a two-way infinite Eulerian path, while $H_2$ does not. We will construct each $G_e$ in such a way that, up to a finite even subgraph, $G_e=H_1$ for $e\in X$, and $G_e=H_2$ for $e\not\in X$. Now, by Theorem \ref{shoenfield}, we know there exists a computable approximation $x$ of $X$. One then starts to build $G_e$ by copying the graph $H_2$ as long as $x$ believes that $e$ does not belong to $X$. Whenever $x$ changes its mind on $e$, we create a cycle and start copying the graph $H_1$ instead as long as $x$ believes that $e$ does belong to $X$. We then continue in this fashion: since $X$ is a $\Delta_2^0$ set we are sure that $x$ eventually stabilizes on $e$, meaning that $G_e$ will eventually look (up to a finite even subgraph) like $H_1$ in case $e \in X$, or like $H_2$ otherwise.
    \end{proof}
\section*{Acknowledgements}
N. Carrasco-Vargas was supported by a grant from the Priority Research Area SciMat under the Strategic Programme Excellence Initiative at Jagiellonian University. V. Delle Rose was supported by the project PRIN 2022 “Logical Methods in Com-
binatorics” No. 2022BXH4R5 of the Italian Ministry of University and Research
(MIUR). C. Rojas was supported by ANID/FONDECYT Regular 1230469 and
ANID/Basal National Center for Artificial Intelligence CENIA FB210017
\bibliographystyle{abbrv}
\bibliography{sep}
\end{document}